\documentclass{article}
\usepackage[english]{babel}
\usepackage[T1]{fontenc}
\usepackage[utf8]{inputenc}
\usepackage{amsthm}
\usepackage{amsmath}
\usepackage{amssymb}
\usepackage{palatino}
\usepackage{graphicx}
\usepackage{dsfont}
\usepackage[all]{xy}
\usepackage[top=2.5cm,bottom=2.5cm,left=2.5cm,right=2.5cm]{geometry}
\usepackage{enumitem}
\usepackage{mathtools}
\usepackage{subcaption}
\usepackage{wrapfig}

\newcommand{\R}{\mathds{R}}

\newcommand{\C}{\mathcal{C}}

\theoremstyle{plain}
\newtheorem{defi}{Definition}
\newtheorem{thm}{Theorem}
\newtheorem{lem}{Lemma}

\newtheorem{rmk}{Remark}

\renewenvironment{proof}{{\bfseries Proof.}}{\qed}

\title{\textbf{Dynamics of neural networks with elapsed time model and learning processes}}
\author{Delphine Salort\footnote{Sorbonne Université, Laboratoire de Biologie Computationnelle et Quantitative.} , Nicolas Torres\footnote{Sorbonne Université, Laboratoire Jacques-Louis Lions.}}
\date{}

\begin{document}
\maketitle
\begin{abstract}
We introduce and study a new model of interacting neural networks, incorporating  the spatial dimension (e.g. position of neurons across the cortex) and some learning processes. The dynamic of each neural network is described via the elapsed time model, that is, the neurons are described by the elapsed time since their last discharge and the chosen learning processes are essentially inspired from the Hebbian rule. We then obtain a system of integro-differential equations, from which we analyze the convergence to stationary states by the means of entropy method and Doeblin’s theory in the case of weak interconnections. We also consider the situation where neural activity is faster than the learning process and give conditions where one can approximate the dynamics by a solution with a similar profile of a steady state. For stronger interconnections, we present some numerical simulations to observe how the parameters of the system can give different behaviors and pattern formations. 
\end{abstract}


\section{Introduction}
The study and modeling of neural networks have been expanded significantly in the past years and still lead to several stimulating open problems. In the case of homogeneous networks, evolution equations describing neural assemblies derived from stochastic processes and microscopic models have become a very active area. Among them, the elapsed time model, has known a growth interest and has been studied by several authors such as Cañizo et al. in \cite{canizo2019asymptotic},Chevalier et al. in \cite{chevallier2015microscopic}, Ly et al. in \cite{ly2009spike}, Mischler et al. in \cite{mischler2018} and Pakdaman et al. in \cite{PPD,PPD2,PPD3}. In particular, the work of Chevalier et al. in \cite{chevallier2015microscopic} establishes a bridge between Poisson point processes that model spike trains and the time elapsed model.

However, the incorporation of spatial dimension, using  those homogeneous models for each unit has not been investigated much yet. Recent works of J. Crevat et al. in \cite{crevat2019diffusive,crevat2019mean,crevat2019rigorous} consider the case with spatial dimension, where each neuron is described via a kinetic PDE derived from FitzHugh-Nagumo model. Else, the main models used for the incorporation of space variable via integro-differential equations are inspired from the Wilson-Cowan \cite{wilson1972excitatory} and Amari \cite{amari1977dynamics} models, where several theoretical and numerical results has been obtained, see Faye et al. in \cite{faye2013existence,faye2010some,faye2013localized}.

Here, we consider the evolution of interacting neural networks, where each neural network is governed by the time elapsed model and has a position $x \in \Omega$, where $\Omega$ is a bounded domain of $\R^d$ (with $d$ the dimension), which models the cortex. Neurons undergo some charging process and then a sudden discharge takes place in response to certain stimulus and this causes other neighboring neurons to discharge, depending on the strength of interconnections in the network. The time variations of these interconnections determine the learning process of the neural network. For simplicity we assume that for each position $x$ we have a homogeneous network that is considered as a single neuron.

Let $n=n(t,s,x)$ be the probability density of finding a neuron at time $t$, such that the elapsed time since its last discharge is $s\ge0$ and its position is $x\in\Omega$. We model the neural network through the following nonlinear renewal system
\begin{equation}
\label{eq0}
\left\{
\begin{matrix*}[l]
\partial_t n(t,s,x)+\partial_s n(t,s,x)+p(s,S(t,x))n(t,s,x)=0&t>0,s>0,x\in\Omega\vspace{0.15cm},\\
N(t,x)\coloneqq n(t,s=0,x)=\int_0^\infty p(s,S(t,x))n(t,s,x)\,ds&t>0,x\in\Omega\vspace{0.15cm},\\
S(t,x)=\int_\Omega w(t,x,y)N(t,y)dy+I(t,x)&t>0,x\in\Omega\vspace{0.15cm},\\
\partial_t w(t,x,y)= -w(t,x,y) + \gamma G(N(t,x),N(t,y))&t>0,\,x,y\in\Omega\vspace{0.15cm},\\
n(t=0,s,x)=n_0(s,x)\ge0,\:w(t=0,x,y)=w_0(x,y)\ge0&s\ge0,\,x,y\in\Omega.
\end{matrix*}
\right.
\end{equation}
The equation for $n$ and the integral boundary condition correspond to the renewal equation, where the function $p\colon[0,\infty)\times\R\to\R$ represents the firing rate of neurons. This function $p$ depends on the elapsed time $s$ and $S(t,x)$, which is the amplitude of stimulation received by the network at time $t$ and position $x$, and we denote $I(t,x)$ an external input.
We say that the system is inhibitory (resp. ) if $p$ is decreasing (resp. increasing) with respect to $S$. 

For the firing rate $p$, we deal with the two following cases.
\begin{subequations}
	\begin{equation}
	\label{lbp1}
 	p_*\le p\le p_\infty,\:\textrm{for some constants}\:p_*,p_\infty>0.
	\end{equation}
	\begin{equation}
	\label{lbp2}
	p_{*}\mathds{1}_{\{s>s_*\}}\le p\le p_{\infty},\:\textrm{for some constants} \:p_*,p_\infty,s_*>0.
	\end{equation}
\end{subequations}
The hypothesis \eqref{lbp2} is an extension of \eqref{lbp1}, since it allows $p$ to vanish for values of $s$ lying on some interval. We mainly deal with the case \eqref{lbp2} in subsection \ref{doeblin}. A special example is to consider
\begin{equation}
\label{pjump}
    p=p_\infty\mathds{1}_{\{s>\sigma(S)\}}
\end{equation}
where $p_\infty>0$ is a constant and $\sigma\colon[0,\infty)\to[0,\infty)$ is a bounded and Lipschitz function. This means that neurons fire if the elapsed time attains the value $\sigma(S)$. In this article we mostly deal with the case when $p$ is smooth, but the results are also valid for functions as in example \eqref{pjump}.

The function $N(t,x)$ is the activity of a neuron at time $t$ and position $x$. This corresponds to integrate with respect to $s$ the term with firing rate in the first equation of \eqref{eq0}. The integral boundary condition of $n$ at $s=0$, states that the elapsed time is reset to zero after a discharge.

The function $w\in\C_b([0,\infty)\times\Omega\times\Omega)$ is the connectivity kernel, which depends on the location of neurons. The third equality of \eqref{eq0} establishes that the amplitude of stimulation received by the network is the result of connectivity among discharging neurons plus the external input $I\in\C_b([0,\infty)\times\Omega)$.

\begin{figure}[ht!]
	\centering
	\includegraphics[width=0.36\linewidth]{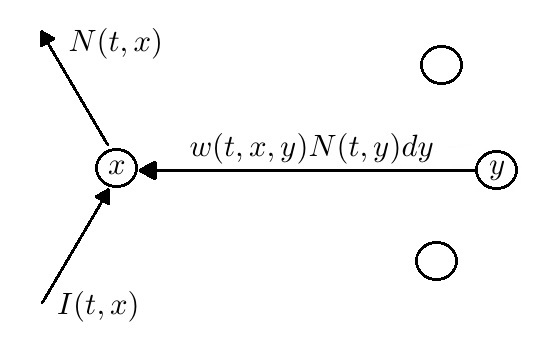}
	\caption{\small A neuron located at position $x$ discharges and sends $N(t,x)$ to rest of the network. At the same time this neuron in $x$ receives $I(t,x)$ from the external input and $w(t,x,y)N(t,y)dy$ from a discharging neuron located at $y$.}
\end{figure}

Furthermore this kernel evolves in time following a learning rule that depends on the smooth function $G\colon\R^2\to\R$ and the activity $N$ at locations $x,y$. Without loss of generality, we assume for simplicity in computations throughout this article that $G$ that satisfies
\begin{equation}
\label{Gnormal}
\|G\|_\infty+\|\nabla G\|_\infty\le1.
\end{equation}

The impact of the learning is studied in the fourth equation of \eqref{eq0}, where $\gamma>0$ is called the connectivity parameter. If $\gamma$ and $\|\tfrac{\partial p}{\partial S}\|_\infty$ are small, we say that the system \eqref{eq0} is under a weak interconnection regime. 

As an example of a learning rule we have 
$$G(N(t,x),N(t,y))=N(t,x)N(t,y),$$
inspired from the Hebbian learning which has been introduced by Hebb in his seminal work in \cite{hebborganization}. This means that if two neurons have simultaneously high activity their connection becomes stronger. Mathematical formulations of this rule have been studied for example by Gerstner and Kistler in \cite{gerstner2002spiking}.

Another example is to take $$G(N(t,x),N(t,y))=\phi(N(t,x)N(t,y))\exp\left(-(N(t,x)-N(t,y))^2\right),$$ with $\phi$ a sigmoid function. This is inspired  from the works of Abbassian et al. in \cite{abbassian2012neural} and Amari in \cite{amari1977dynamics} on neural fields and membrane potentials. This means that the interconnection of two neurons becomes stronger if their activities are similar and large enough.

Other learning models have been studied in neural networks. In the work of Perthame et al. in \cite{perthame2017distributed}, they studied the learning process for the leaky integrate-and-fire model (for references about this model, see \cite{caceres2011analysis,carrillo2015qualitative}). They indirectly generalize the Hebbian learning via distributed synaptic weights, which means that there is a total activity distributed throughout the network according to some parameter. In contrast, we present a learning model for the time elapsed dynamics that can generalize directly the Hebbian model via evolution of the connectivity kernel.

Finally, $(n_0,w_0)$ denotes the initial configuration of the system with
\begin{equation}
\label{inidata}
n_0\in\C_b(\Omega,L^1_s),\quad w_0\in\C_b(\Omega\times\Omega).
\end{equation} 
Observe that for each $x\in\Omega$ the $L^1$-norm of $n(t,\cdot,x)$ is formally preserved, i.e. there exists $g\in\C_b(\Omega)$ non-negative such that
\begin{equation}
\label{gx}
g(x)\coloneqq\int_0^\infty\hspace{-0.2cm}n_0(s,x)\,ds=\int_0^\infty \hspace{-0.2cm}n(t,s,x)\,ds\\ \ge0\quad\forall t>0,\,x\in\Omega,\quad \int_\Omega g(x)\,dx=1.
\end{equation}

The rest of the article is organized as follows. In section \ref{wellp} we prove that system \eqref{eq0} is well-posed in a suitable space when the interconnections are weak. Under the same regime of connectivity, we prove in section \ref{stationary} the existence of stationary states and in section \ref{conv} we prove the exponential convergence to equilibrium in two different ways: via the entropy method and via Doeblin's theory. Furthermore in section \ref{slowlearning} we study a variant of system \eqref{eq0} where the time scale for learning is much slower than that of elapsed time dynamics. Finally in section \ref{numerical} we present some examples of numerical simulations for different external inputs, connectivity parameters and learning rules.

\section{Well-posedness for the weak interconnection case}
\label{wellp}
We prove that system \eqref{eq0} is well-posed under the weak interconnection regime. In order to do so, we start by studying an auxiliary linear problem where the amplitude of stimulation is fixed and then we proceed to prove well-posedness of system \eqref{eq0} via a contraction argument.
\subsection{The linear problem}
Given $S\in\C_b([0,\infty)\times\Omega)$, we consider the following linear problem
\begin{equation}
\label{eql}
\left\{
\begin{matrix*}[l]
\partial_t n+\partial_s n+p(s,S(t,x))n=0& t>0,s>0,x\in\Omega\vspace{0.15cm},\\
N(t,x)\coloneqq n(t,s=0,x)=\int_0^\infty p(s,S(t,x))n\,ds& t>0,x\in\Omega\vspace{0.15cm},\\
n(t=0,s,x)=n_0(s,x)\ge0& s\ge0,x\in\Omega.
\end{matrix*}
\right.
\end{equation}
We look for weak solutions satisfying $n\in\C_b([0,\infty)\times\Omega,L^1_s)$, so that $N\in\C_b([0,\infty)\times\Omega)$. Furthermore, in this linear system the variable $x$ is simply a parameter, since there is no derivative or integral term involving the position. 
\begin{lem}
\label{linear}
Assume that $n_0\in\C_b(\Omega,L^1_s)$ and $p\in W^{1,\infty}((0,\infty)\times\R)$ satisfies \eqref{lbp2}. Then for a given $S\in\C_b([0,\infty)\times\Omega)$, the equation \eqref{eql} has a unique weak solution $n\in\C_b([0,\infty)\times\Omega,L^1_s)$ with $N\in\C_b([0,\infty)\times\Omega)$. Moreover $n$ is non-negative and verifies the property \eqref{gx}.
\end{lem}
In particular this lemma proves the property \eqref{gx} for the non linear system \eqref{eq0}. Moreover, this lemma is also valid for $p$ defined in \eqref{pjump} with a similar proof.

\begin{proof}
	We start by noticing that a solution of the linear system \eqref{eql} satisfies the following fixed point equation
	\begin{equation}
	\label{fixnl}
	\begin{split}
	n(t,s,x)=\Psi[n](t,s,x) &\coloneqq n_0(s-t,x)\exp\left(-\int_0^t p(\tau+s-t,S(\tau,x))\,d\tau\right)\mathds{1}_{\{ s>t\}}\\
	&\quad +N(t-s,x)\exp\left(-\int_0^s p(\tau,S(t-s+\tau,x))\,d\tau\right)\mathds{1}_{\{0<s<t\}},
	\end{split}
	\end{equation}
	with $N(t,x)=\int_0^\infty p(u,S(t,x))n(t,u,x)\,du$, which depends on $n$.
	
	Let $T>0$ and $X_T\coloneqq\{ n\in\C_b([0,T]\times\Omega,L^1_s)\colon n(0)=n_0\}$, it readily follows that $\Psi$ maps $X_T\to X_T$. We prove by the contraction principle that $\Psi$ has a unique fixed point in $X_T$ for $T>0$ small enough, i.e. there exists a unique weak solution of \eqref{eql} defined on $[0,T]$. Consider $n_1,n_2\in X_T$ so we have
	\begin{equation}
	\begin{split}
	\int_0^\infty|\Psi[n_1]-\Psi[n_2]|(t,s,x)\,ds&\le\int_0^t|N_1(t-s,x)-N_2(t-s,x)|\,ds\\
	&\le T\sup_{(t,x)\in[0,T]\times\Omega}|N_1-N_2|(t,x)\\
	&\le T\,p_{\infty}\sup_{(t,x)\in[0,T]\times\Omega}\|n_1(t,x)-n_2(t,x)\|_{L^1_s},
	\end{split}
	\end{equation}
	thus for $T<\frac{1}{p_{\infty}}$, we have proved that $\Psi$ is a contraction and there exists a unique $n\in X_{T}$ such that $\Psi[n]=n$. Since the choice of $T$ is independent of $n_0$, we can reiterate this argument to get a unique solution of \eqref{eql}, which is defined for all $t\ge0$.
	
	Next we prove the mass conservation property. Since $n$ satisfies the fixed point equation in \eqref{fixnl}, it also verifies the following equality
	\begin{equation}
	\label{soln}
	n(t,s,x)=n_0(s-t,x)\mathds{1}_{\{ s>t\}}-\int_0^t p(s-t+\tau,S(\tau,x))n(\tau,s-t+\tau,x)\mathds{1}_{\{ s>t-\tau\}}\,d\tau+N(t-s,x)\mathds{1}_{\{0<s<t\}},
	\end{equation}
	hence we get the property by integrating with respect to $s$ on $(0,\infty)$.
	
	Finally, since $n_0$ is non-negative then $\Psi$ preserves positivity, so by uniqueness of fixed point the corresponding solution $n$ must be non-negative.	
\end{proof}

\subsection{The non-linear problem}
We are now ready to prove that system \eqref{eq0} is well-posed in the case of weak interconnection.
\begin{thm}[Well-posedness for weak interconnections]
	Assume \eqref{inidata}-\eqref{gx} and that $p\in W^{1,\infty}((0,\infty)\times\R)$ satisfies \eqref{lbp2}. Then for 
	$$\|g\|_\infty|\Omega|\,\|\tfrac{\partial p}{\partial S}\|_\infty\max\left\{\|w_0\|_\infty,\gamma\right\}<1,$$ 
	the system \eqref{eq0} has a unique solution with $n\in\C_b([0,\infty)\times\Omega,L^1_{s}),\,N\in\C_b([0,\infty)\times\Omega),\,S\in\C_b([0,\infty)\times\Omega)$ and $w\in\C_b([0,\infty)\times\Omega\times\Omega)$. Moreover, $n$ is non-negative for all $t>0$.
\end{thm}
\begin{proof}
	Consider $T>0$. We fix a function $S\in\C_b([0,\infty)\times\Omega)$ and define the functions $n\in\C_b([0,\infty)\times\Omega,L^1_s),\,N\in\C_b([0,\infty)\times\Omega)$ which are solutions of \eqref{eql} by lemma \ref{linear}. Furthermore, the solution of this linear system preserves positivity and the condition \eqref{gx}.
	
		The solution  $w\in\C_b([0,\infty)\times\Omega\times\Omega)$ is obtained through the formula
	\begin{equation}
	\label{solw}
	w(t,x,y)=e^{-t}w_0(x,y)+\gamma\int_0^te^{-(t-\tau)}G(N(\tau,x),N(\tau,y))\,d\tau.
	\end{equation}
	So we have a solution of system \eqref{eq0} defined on $[0,T]$ if $S$ satisfies for all $0\le t\le T$ and $x\in\Omega$, the following fixed point condition
	\begin{equation}
	\label{fixS}
		S(t,x)=\mathcal{T}[S](t,x)\coloneqq \int w(t,x,y)\left(\int_0^\infty p(s,S(t,y))n(t,s,y)\,ds\right)dy+I(t,x).
	\end{equation}
	We prove that $\mathcal{T}$ defines for all $T>0$ an operator that maps $X_T\to X_T$ with $X_T\coloneqq\C_b([0,T]\times\Omega)$. First, we observe the following estimate for the activity
	\begin{equation}
	\label{bdN}
		|N(t,x)|\le p_\infty\|g\|_\infty,\quad\forall(t,x)\in [0,T]\times\Omega.	
	\end{equation}
	And from the equation of $w$, we get the following uniform
	\begin{equation}
	\label{bdw}
		|w(t,x,y)|\le \max\{\|w_0\|_\infty,\gamma\},\quad\forall(t,x,y)\in [0,T]\times\Omega\times\Omega.
	\end{equation}
	Let $A\coloneqq\max\{\|w_0\|_\infty,\gamma\}$. This implies that for any $S\in X_T$ we have
	$$\|\mathcal{T}[S]\|_\infty\le A p_{\infty}+\|I\|_\infty,$$
	and it is immediate that $\mathcal{T}[S]$ is a continuous function, thus $\mathcal{T}[S]\in X_T$. 
	
		We now prove that for $T$ small enough, $\mathcal{T}$ is a contraction. Consider $S_1,\,S_2\in X_T$ and observe that the difference between $w_1$ and $w_2$ satisfies, by using \eqref{solw},
	\begin{equation}
	\label{diffw}
		|w_1(t,x,y)-w_2(t,x,y)|\le2\gamma T\|N_1-N_2\|_\infty.
	\end{equation}	
	Next, for the difference between $N_1$ and $N_2$ we have
	\begin{equation}
	\label{diffN}
		\begin{split}
		|N_1-N_2|(t,x) &\le\int_0^\infty|p(s,S_1)\,n_1-p(s,S_2)\,n_2|\,ds\\
		&\le \int_0^\infty|p(s,S_1)-p(s,S_2)|\,n_1
		\,ds+\int_0^\infty p(s,S_2)|n_1-n_2|
		\,ds\\
		&\le \|g\|_\infty\,\|\tfrac{\partial p}{\partial S}\|_\infty\,\|S_1-S_2\|_\infty+p_{\infty}\|n_1-n_2\|_{L^\infty_{t,x}L^1_s}.
		\end{split}
	\end{equation}
	Now we have to estimate the difference between $n_1$ and $n_2$. From \eqref{soln} and estimate \eqref{diffN}, we get
	\begin{equation*}
	\|n_1-n_2\|_{L^\infty_{t,x}L^1_s}\le 2T\|g\|_\infty\,\|\tfrac{\partial p}{\partial S}\|_\infty\,\|S_1-S_2\|_\infty+2Tp_{\infty}\|n_1-n_2\|_{L^\infty_{t,x}L^1_s}.
	\end{equation*}
	Then, for $T<\tfrac{1}{2p_{\infty}}$ we obtain
	\begin{equation}
	\|n_1-n_2\|_{L^\infty_{t,x}L^1_s}\le\frac{2T\|g\|_\infty\|\tfrac{\partial p}{\partial S}\|_\infty}{1-2Tp_{\infty}}\|S_1-S_2\|_{\infty}.
	\end{equation}
	Finally by combining the estimates \eqref{bdN}-\eqref{diffw}, the operator $\mathcal{T}$ satisfies
	\begin{equation}
	\begin{split}
	|\mathcal{T}[S_1]-\mathcal{T}[S_2]|&\le\int|w_1-w_2|\,N_1\,dy+\int|w_2|\,|N_1-N_2|\,dy\\
	&\le \big( 2\gamma T p_{\infty}+|\Omega|A\big)\|N_1-N_2\|_{\infty}\\
	&\le C\,\|S_1-S_2\|_{\infty},
	\end{split}
	\end{equation}
	with $C>0$ given by
	$$C\coloneqq\|g\|_\infty\,\|\tfrac{\partial p}{\partial S}\|_\infty\left(2\gamma T p_{\infty}+|\Omega|A\right)\left(1+\frac{2Tp_\infty}{1-2Tp_{\infty}}\right).$$
	Hence for $\|g\|_\infty|\Omega|\,\|\tfrac{\partial p}{\partial S}\|_\infty A<1$ and $T$ small enough we get $C<1$, so $\mathcal{T}$ is a contraction.

	From Picard's fixed point we get a unique $S\in X_T$ such that $\mathcal{T}[S]=S$, and this implies the existence of a unique solution of \eqref{eq0} defined on $[0,T]$.
	Since estimates \eqref{bdN} and \eqref{bdw} are uniform in $T$, we can iterate this argument to get a unique solution of \eqref{eq0} defined for all $t>0$.
	
	Furthermore, we conclude from this construction that the non-linear system \eqref{eq0} preserves positivity and satifisfies \eqref{gx} like the linear system \eqref{eql}.
\end{proof}

\begin{rmk}
    From estimate \eqref{bdN}, we observe that we only need the function $G$ to be bounded on the set $[0,p_\infty\|g\|_\infty]^2$. This justifies that we do not lose generality in assuming $G$ normalized according to \eqref{Gnormal}.
\end{rmk}

The condition on $p$ can be relaxed to wider class of functions, as we see in the following example.
\begin{thm}
	Consider $p$ defined in \eqref{pjump}. Assume in addition that $n_0\in L^\infty_{s,x}$ and $w_0\in C_b(\Omega\times\Omega)$, then the same result holds if $$p_\infty\|\sigma'\|_\infty|\Omega|\max\left\{\|w_0\|_\infty,\gamma\right\}(\|n_0\|_\infty+p_{\infty}\|g\|_\infty)<1.$$
\end{thm}

\begin{proof}
	The proof is the same as for the previous theorem. Let $\mathcal{T}$ be the operator defined before, we have to verify the contraction principle. The estimates \eqref{bdN}-\eqref{diffw} for $N$ and $w$ remain unchanged. 
	
	Now, from the solution of linear problem \eqref{eql} we get for $n$ the uniform estimate 
	$$|n(t,s,x)|\le \|n_0\|_\infty+p_{\infty}\|g\|_\infty,\:\forall (t,s,x)\in [0,T]\times(0,\infty)\times\Omega.$$
 	Let $A\coloneqq\max\left\{\|w_0\|_\infty,\gamma\right\}$ and $B\coloneqq\|n_0\|_\infty+p_{\infty}\|g\|_\infty$. In this case the difference between $N_1$ and $N_2$ in \eqref{diffN} is replaced by
	\begin{equation*}
	\begin{split}
	|N_1-N_2|(t,x)
	&\le \int_0^\infty|p(s,S_1)-p(s,S_2)|\,n_1\,ds+\int_0^\infty p(s,S_2)|n_1-n_2|\,ds\\
	&\le p_\infty\left|\int_{\sigma(S_1)}^{\sigma(S_2)}n_1\,ds\right|+\int_0^\infty p(s,S_2)|n_1-n_2|\,ds\\
	&\le p_\infty\|\sigma'\|_\infty B\|S_1-S_2\|_\infty+p_{\infty}\|n_1-n_2\|_{L^\infty_{t,x}L^1_s}.
	\end{split}
	\end{equation*}
	And from \eqref{soln}, the difference between $n_1$ and $n_2$ satisfies
	$$\|n_1-n_2\|_{L^\infty_{t,x}L^1_s}\le 2Tp_\infty\|\sigma'\|_\infty B\|S_1-S_2\|_\infty+2Tp_\infty\|n_1-n_2\|_{L^\infty_{t,x}L^1_s}.$$
	Then, for $T<\tfrac{1}{2p_\infty}$ we conclude similarly
	$$\|n_1-n_2\|_{L^\infty_{t,x}L^1_s}\le\frac{2Tp_\infty\|\sigma'\|_\infty B}{1-2Tp_{\infty}}\|S_1-S_2\|_{\infty}.$$
	Hence, by combining the estimates for $N_1-N_2$ and $n_1-n_2$, the operator $\mathcal{T}$ verifies
	\begin{equation*}
	\begin{split}
	|\mathcal{T}[S_1]-\mathcal{T}[S_2]|&\le\int|w_1-w_2|\,N_1\,dy+\int|w_2|\,|N_1-N_2|\,dy\\
	&\le \big( 2\gamma T p_{\infty}+|\Omega|A\big)\|N_1-N_2\|_{\infty}\\
	&\le C\,\|S_1-S_2\|_{\infty},
	\end{split}
	\end{equation*}
	with $C>0$ given by
	$$C\coloneqq p_\infty\|\sigma'\|_\infty B\left(2\gamma T p_{\infty}+|\Omega|A\right)\left(1+\frac{2Tp_\infty}{1-2Tp_\infty}\right).$$
	Thus for $p_\infty\|\sigma'\|_\infty|\Omega|AB<1$ and $T$ small enough we get that $\mathcal{T}$ is a contraction and this implies the existence of a unique solution defined on $[0,T]$. Finally, we can iterate this argument to get a unique globally defined solution, like we asserted in the previous theorem.
\end{proof}

\section{Stationary states}
\label{stationary}
Assume the input $I$ depends only on position. We now study the stationary solutions of \eqref{eq0}, i.e. the system given by
\begin{equation}
\label{eq1}
\left\{
\begin{matrix*}[l]
\partial_s n(s,x)+p(s,S(x))n(s,x)=0& s>0,x\in\Omega\vspace{0.15cm},\\
N(x)\coloneqq n(s=0,x)=\int_0^\infty p(s,S(x))n(s,x)\,ds& x\in\Omega\vspace{0.15cm},\\
S(x)=\int_\Omega w(x,y)N(y)dy+I(x)& x\in\Omega\vspace{0.15cm},\\
w(x,y)=\gamma G(N(x),N(y))&x,y\in\Omega,\\
\end{matrix*}
\right.
\end{equation}
where $n\in L^1_{s,x},\:N,S\in C_b(\Omega)$ and $w\in\C_b(\Omega\times\Omega)$.

If the amplitude $S$ is given, we can determine $n,N$ and $w$ through the formulas
\begin{equation}
\label{solest}
\begin{matrix*}[l]
n(s,x)=N(x)e^{-\int_0^sp(\tau,S(x))\,d\tau}\vspace{0.15cm},\\
N(x)=g(x)\left(\int_0^\infty e^{-\int_0^up(\tau,S(x))\,d\tau}\,du\right)^{-1}\vspace{0.15cm},\\
w(x,y)=\gamma G\big(g(x)\,F(S(x)),g(y)\,F(S(y))\big).
\end{matrix*}
\end{equation}
We define $F\colon\R\to\R_+$ given by
\begin{equation}
	\label{CapF}
	F(S)\coloneqq\left(\int_0^\infty e^{-\int_0^sp(\tau,S)\,d\tau}\,ds\right)^{-1},
\end{equation}
and we get that $(n,N,S,w)$ in \eqref{solest} corresponds to a stationary solution of \eqref{eq0} if $S$ satisfies the following fixed point condition
\begin{equation}
\label{fixSeq}
	S(x)=\mathcal{T}[S](x)\coloneqq\gamma\int G\big(\,g(x)F(S(x)),\,g(y)F(S(y))\,\big)\,g(y)F(S(y))\,dy+I(x).
\end{equation}

The following result asserts that there exists a unique steady state for a given $g\in\C_b(\Omega)$, under weak interconnection regime.
\begin{thm}
\label{thmest}
Assume that $p\in W^{1,\infty}((0,\infty)\times\R)$ satisfies \eqref{lbp2} and $g\in\C_b(\Omega)$. For $\gamma$ small enough, the system \eqref{eq0} has a unique stationary state $(n^*,N^*,S^*,w^*)$, with $n^*\in \C_b(\Omega,L^1_s)$ satisfying $\int_0^\infty n^*(s,x)\,ds=g(x)$ and $N^*\in \C_b(\Omega),\,w^*\in\C_b(\Omega\times\Omega)$, which are determined by a unique amplitude of stimulation $S^*\in\C_b(\Omega)$ satisfying $\mathcal{T}[S^*]=S^*$.
\end{thm}

To prove the result we use the following lemma about the function $F$.

\begin{lem}
	Under the hypothesis of theorem \ref{thmest}, $F$ is a bounded and Lipschitz function.
\end{lem}
\begin{proof}
	It readily follows that $F$ is bounded since it satisfies the following estimate
	$$0<F(S)\le \left(\int_0^{\infty}e^{-p_{\infty} s}\,ds\right)^{-1}=p_{\infty}.$$
	On the other hand, $F'$ is given by the formula
	$$F'(S)=F(S)^2\left[\int_0^ \infty e^{-\int_0^s p(\tau,S)\,d\tau}\left(\int_0^s\frac{\partial p}{\partial S}(\tau,S)\,d\tau\right)ds\right],$$
	so we have the following estimate
	\begin{equation*}
	\begin{split}
	|F'(S)|&\le p_{\infty}^2\left\|\frac{\partial p}{\partial S}\right\|_\infty\left[\int_0^\infty e^{-\int_0^s p(\tau,S)\,d\tau}s\,ds\right]\\
	&\le p_{\infty}^2\left\|\frac{\partial p}{\partial S}\right\|_\infty \left[\int_0^\infty e^{-p_{*}(s-s_*)_+}s\,ds\right]\\
	&\le p_{\infty}^2\left\|\frac{\partial p}{\partial S}\right\|_\infty\left[\frac{s_*^2}{2}+\frac{s_*}{p_*}+\frac{1}{p_*^2}\right].
	\end{split}
	\end{equation*}
	Hence $F$ is Lipschitz.
\end{proof}

\begin{rmk}
	In the case of $p$ defined in \eqref{pjump} we get
	$$F(S)=\frac{1}{p_{\infty}^{-1}+\sigma(S)},$$
	so $F$ bounded and Lipschitz since $\sigma$ is. Hence the theorem is also valid for this case.
\end{rmk}
Next, we conclude the proof of our main theorem.

\begin{proof}
	It is straightforward that in \eqref{fixSeq} $\mathcal{T}$ defines an operator that maps $\C_b(\Omega)\to\C_b(\Omega)$. Since $F$ is bounded and Lipschitz we get for $S_1,S_2\in\C_b(\Omega)$
	\begin{equation*}
	\begin{split}
	|\mathcal{T}[S_1]-\mathcal{T}[S_2]|(x)&\le2\gamma\|g\|_\infty\|F\|_\infty\|F'\|_\infty\|S_1-S_2\|_\infty\\
	&\quad+\gamma\|F'\|_\infty\|S_1-S_2\|_\infty, 
	\end{split}
	\end{equation*}
    Thus for $\gamma$ satisfying  $$\gamma\|F'\|_\infty\big(2\|g\|_\infty\|F\|_\infty+1\big)<1,$$
	the operator $\mathcal{T}$ is a contraction and there exists a unique $S^*\in\C_b(\Omega)$ such that $\mathcal{T}[S^*]=S^*$. Therefore we get a unique stationary state determined through the formulas in \eqref{solest}.
\end{proof}

\section{Convergence to equilibrium}
\label{conv}
Our next result about system \eqref{eq0} is the convergence to equilibrium when $t\to\infty$, under the weak interconnection regime i.e. with $\gamma$ and $\|\tfrac{\partial p}{\partial S}\|_\infty$ small enough. For the proof of this result we present two different approaches: the relative entropy method and the Doeblin theory applied to stochastic semi-groups.
\subsection{Entropy method approach}
Firstly we prove the convergence result when the firing rate $p$ is strictly positive by means of the relative entropy method studied in \cite{michel2005general,perthame2006transport} and following the ideas in \cite{kang2015}.
 
\begin{thm}[Long term behavior for the weak interconnection regime]
	\label{conveq0}
	 Assume \eqref{inidata}-\eqref{gx} and that $p\in W^{1,\infty}((0,\infty)\times\R)$ satisfies \eqref{lbp1}. For $\gamma$ and $\|\tfrac{\partial p}{\partial S}\|_\infty$ small enough let $(n^*,N^*,S^*,w^*)$ be the corresponding stationary state of \eqref{eq0}. Then there exist $C,\lambda>0$ such that the solution of \eqref{eq0} satisfies
	\begin{equation}
	\label{convn}
		\|n(t)-n^*\|_{L^1_{s,x}}+\|w(t)-w^*\|_{L^1_{x,y}}\le Ce^{-\lambda t}\left(\|n_0-n^*\|_{L^1_{s,x}}+\|w_0-w^*\|_{L^1_{x,y}}\right),\:\forall t\ge 0.
	\end{equation}
	Moreover $\|S(t)-S^*\|_{L^1_x}$ and $\|N(t)-N^*\|_{L^1_x}$ converge exponentially to $0$ when $t\to\infty$.
\end{thm}

In other words, if interconnections are weak then solutions converge exponentially to equilibrium.

\begin{proof}
	Observe that $n-n^*$ and $w-w^*$ satisfy
	\begin{equation*}
	\begin{matrix}
	\partial_t(n-n^*)+\partial_s(n-n^*)+p(s,S)(n-n^*)=-(p(s,S)-p(s,S^*))n^*,\vspace{0.15cm}\\
	\partial_t(w-w^*)=-(w-w^*)+\gamma G(N(t,x),N(t,y))-\gamma G(N^*(x),N^*(y)),
	\end{matrix}
	\end{equation*}
	so we have the following inequalities
	\begin{equation*}
	\begin{matrix}
	\partial_t|n-n^*|+\partial_s|n-n^*|+p(s,S)|n-n^*|\le \left\|\tfrac{\partial p}{\partial S}\right\|_\infty |S-S^*|\,n^*,\vspace{0.15cm}\\
	\partial_t|w-w^*|\le-|w-w^*|+\gamma\big(|N(t,x)-N^*(x)|+|N(t,y)-N^*(y)|\big),
	\end{matrix}
	\end{equation*}
    By integrating with respect to the corresponding variables we get
	\begin{equation}
	\label{intpuce}
	\begin{matrix*}[l]
	\displaystyle
	\frac{\partial}{\partial t}\iint_0^\infty|n-n^*|\,ds\,dx+\iint_0^\infty p(s,S)|n-n^*|\,ds\,dx\le\int|N-N^*|\,dx+\|g\|_\infty\left\|\tfrac{\partial p}{\partial S}\right\|_\infty\int|S-S^*|\,dx,\vspace{0.15cm}\\
	\displaystyle
	\frac{\partial}{\partial t}\iint |w-w^*|\,dx\,dy\le-\iint |w-w^*|\,dx\,dy+2\gamma|\Omega|\,\int|N-N^*|\,dx.	
	\end{matrix*}
	\end{equation}
	Thus we have to estimate the terms in the right-hand side of both inequalities. For the difference between $N$ and $N^*$ we get
	\begin{equation}
	\label{N-N*}
		\int|N-N^*|\,dx\le\|g\|_\infty\left\|\tfrac{\partial p}{\partial S}\right\|_\infty\int|S-S^*|\,dx+\int\left|\int_0^\infty p(s,S)(n-n^*)\,ds\right|dx.
	\end{equation}
	Next, for the difference between $S$ and $S^*$ we obtain
	\begin{equation*}
		\begin{split}
		\int|S-S^*|\,dx&\le\iint w^*|N(t,y)-N^*(y)|\,dx\,dy+\iint N(t,y)|w-w^*|\,dx\,dy\\
		&\le\gamma|\Omega|\,\int|N-N^*|\,dx+p_\infty\|g\|_\infty\iint |w-w^*|\,dx\,dy,
		\end{split}
	\end{equation*}
	Hence from \eqref{N-N*}, the following inequality holds
	\begin{equation*}
	\begin{split}
		\int|S-S^*|\,dx&\le \gamma|\Omega|\,\|g\|_\infty\left\|\tfrac{\partial p}{\partial S}\right\|_\infty\int|S-S^*|\,dx+\gamma|\Omega|\, p_\infty\iint_0^\infty|n-n^*|\,ds\,dx\\
		&\quad+p_\infty\|g\|_\infty\iint |w-w^*|\,dx\,dy,
	\end{split}
	\end{equation*}
	and if $\beta\coloneqq\gamma|\Omega|\,\|g\|_\infty\left\|\tfrac{\partial p}{\partial S}\right\|_\infty<1$, we deduce the following estimate
	\begin{equation}
	\label{S-S*}
	\int|S-S^*|\,dx\le\frac{p_\infty}{1-\beta}\left(\gamma|\Omega|\,\iint_0^\infty|n-n^*|\,ds\,dx+\|g\|_\infty\iint |w-w^*|\,dx\,dy\right).
	\end{equation}
	Thus from \eqref{intpuce} we get
	\begin{equation}
	\begin{split}
	\frac{\partial}{\partial t}\iint_0^\infty|n-n^*|\,ds\,dx&\le-\iint_0^\infty p(s,S)|n-n^*|\,ds\,dx+\int\left|\int_0^\infty p(s,S)(n-n^*)\,ds\right|dx\\
	&\quad+\frac{2p_\infty\|g\|_\infty\|\frac{\partial p}{\partial S}\|_\infty}{1-\beta}\left(\gamma|\Omega|\,\iint_0^\infty|n-n^*|\,ds\,dx+\|g\|_\infty\iint |w-w^*|\,dx\,dy\right)
	\end{split}
	\end{equation}
	Since $\int_0^\infty(n-n^*)\,ds=0$ and $p\ge p_*$ we may use the argument from \cite{mischler2018,perthame2006transport} to get
	\begin{equation*}
	    \begin{split}
	       \int\left|\int_0^\infty p(s,S)(n-n^*)\,ds\right|dx&=\int\left|\int_0^\infty (p(s,S)-p_*)(n-n^*)\,ds\right|dx\\
	       &\le\iint_0^\infty (p(s,S)-p_*)|n-n^*|\,ds\,dx.
	    \end{split}
	\end{equation*}
	Therefore we deduce the following inequality for $n-n^*$
	\begin{equation}
	\begin{split}
	\frac{\partial}{\partial t}\iint_0^\infty|n-n^*|\,ds\,dx&\le-\left(p_*-\frac{2\beta p_\infty}{1-\beta}\right)\iint_0^\infty|n-n^*|\,ds\,dx\\&\quad+\frac{2p_\infty\|g\|_\infty^2\|\frac{\partial p}{\partial S}\|_\infty}{1-\beta}\iint |w-w^*|\,dx\,dy.
	\end{split}
	\end{equation}
	On the other hand from the second inequality in \eqref{intpuce} and estimate \eqref{N-N*} we get for $w-w^*$
	\begin{equation}
	\begin{split}
	\frac{\partial}{\partial t}\iint |w-w^*|\,dx\,dy&\le-\iint |w-w^*|\,dx\,dy+\frac{2\beta p_\infty\|g\|_\infty}{1-\beta}\iint |w-w^*|\,dx\,dy\\
	&\quad+2\gamma|\Omega|\,p_\infty\left(\frac{\beta}{1-\beta}+1\right)\iint_0^\infty|n-n^*|\,ds\,dx.
	\end{split}
	\end{equation}
	If we add these two inequalities we get an expression of the form
	\begin{equation}
	\begin{split}
	\frac{\partial}{\partial t}\left(\iint_0^\infty|n-n^*|\,ds\,dx+\iint |w-w^*|\,dx\,dy\right)&\le-(p_*-C_1)\iint_0^\infty|n-n^*|\,ds\,dx\\
	&\quad-(1-C_2)\iint |w-w^*|\,dx\,dy,
	\end{split}
	\end{equation}
	with $C_1,C_2>0$ given by
	$$\begin{matrix}
	\displaystyle
	C_1=2\gamma|\Omega| p_\infty\left(\frac{\|g\|_\infty\|\frac{\partial p}{\partial S}\|_\infty}{1-\beta}+\frac{\beta}{1-\beta}+1\right),\vspace{0.15cm}\\
	\displaystyle
	C_2=\frac{2\beta p_\infty\|g\|_\infty\|\frac{\partial p}{\partial S}\|_\infty}{1-\beta}\left(1+\gamma|\Omega|\,\right).
	\end{matrix}$$
	If $\gamma$ and $\|\frac{\partial p}{\partial S}\|_\infty$ are such that $C_1<p_*$ and $C_2<1$, we conclude, by solving the corresponding differential inequality, the existence of $C,\lambda>0$ satisfying the estimate \eqref{convn}. Furthermore the convergence of $N,S$ and $w$ readily follows from estimates \eqref{N-N*} and \eqref{S-S*}.   
\end{proof}

\subsection{Doeblin theory approach}
\label{doeblin}
The previous convergence result for the system \eqref{eq0} can be extended when the firing rate $p$ satisfies the hypothesis \eqref{lbp2} for a $s_*>0$ small enough (see \cite{PPD} for an example). We assert that this result is also valid when $p$ satisfies the condition \eqref{lbp2} with any $s_*>0$. In order to improve the convergence, we follow the ideas of Cañizo et al. in \cite{canizo2019asymptotic} to study the asymptotic behavior of the linear system \eqref{P_t} by means of Doeblin's theory.
\subsubsection{The linear case}
Given $S\in\C_b(\Omega)$, we consider the linear problem given by
\begin{equation}
\label{P_t}
\left\{
\begin{matrix*}[l]
\partial_t n+\partial_s n+p(s,S(x))n=0&t>0,s>0,x\in\Omega\vspace{0.15cm},\\
N(t,x)\coloneqq n(t,s=0,x)=\int_0^\infty p(s,S(x))n\,ds& t>0,x\in\Omega\vspace{0.15cm},\\
n(t=0,s,x)=n_0(s,x)&s\ge0,x\in\Omega.
\end{matrix*}
\right.
\end{equation}
From lemma \ref{linear} we know that this system has a unique solution $n\in\C_b([0,\infty)\times\Omega,L^1_{s})$. Since the variable $x$ is just a parameter, for a fixed $x\in\Omega$ we define from equation \eqref{P_t} the stochastic semi-group $P_t\colon L^1_s\to L^1_s$ given by
$$P_t n_0(s,x)=n(t,s,x).$$
A key property on the solutions of this system is the exponential convergence to equilibrium as we state in the following theorem:
\begin{thm}
	\label{DoeblinConv}
	Consider $n_0\in\C_b(\Omega,L^1_s)$ with its corresponding $g\in \C_b(\Omega)$ and that $p$ satisfies \eqref{lbp2}, then there exists a unique stationary solution $n^*$ of equation \eqref{P_t} satisfying $\int_0^\infty n^*(s,x)\,ds=g(x)$. Moreover, the corresponding solution of \eqref{P_t} satisfies
	$$\|n(t,\cdot,x)-n^*(\cdot,x)\|_{L^1_s}\le\frac{1}{1-\alpha}e^{-\lambda t}\|n_0(\cdot,x)-n^*(\cdot,x)\|_{L^1_s}\quad\forall t\ge 0,\,x\in\Omega.$$
	with $\alpha=p_{*}s_*e^{-2p_{\infty}s_*}$ and $\lambda=-\frac{\log(1-\alpha)}{2s_*}>0$.
\end{thm}

For the sake of completeness, we include the proof of this result done by Cañizo et al. in the theorem 3.12 of \cite{canizo2019asymptotic}. In our case, functions have mass $g(x)$ instead of having mass $1$ with respect to $L^1_s$. We start by reminding some concepts on stochastic semi-groups and Doeblin's theorem.
\begin{defi}
	Let $X$ be a measure space and $P_t\colon L^1(X)\to L^1(X)$ be a linear semi-group. We say that $P_t$ is a stochastic semi-group if $P_t f\ge0$ for all $f\ge0$ and $\int_X P_t f=\int_X f$ for all $f\in L^1(X)$. In other words, $(P_t)$ preserves the subset of probability densities $\mathcal{P}(X)$.
\end{defi}	

\begin{defi}
	Let $P_t\colon L^1(X)\to L^1(X)$ be a stochastic semi-group. We say that $(P_t)$ satisfies Doeblin's condition if there exists $t_0>0,\,\alpha\in(0,1)$ and $\nu\in \mathcal{P}(X)$ such that
	$$P_{t_0}f\ge\alpha\nu\quad\forall f\in\mathcal{P}(X).$$
\end{defi}
\begin{thm}[Doeblin's Theorem] 
	Let $P_t\colon L^1(X)\to L^1(X)$ be a stochastic semi-group that satisfies Doeblin's condition. Then the semigroup has a unique equilibrium $n_*$ in $\mathcal{P}(X)$. Moreover, for all $n\in\mathcal{P}(X)$ we have
	$$\|P_t(n-n_*)\|_{L^1(X)}\le\frac{1}{1-\alpha}e^{-\lambda t}\|n-n_*\|_{L^1(X)}\quad\forall t\ge 0,$$
	with $\lambda=-\frac{\ln(1-\alpha)}{t_0}>0$.
\end{thm}

Next, we continue with the proof of theorem \ref{DoeblinConv}.

\begin{proof}
 	Let $n$ be the solution of \eqref{P_t}. For fixed $x\in\Omega$, we claim $n$ satisfies the following inequality
 	\begin{equation}
 	\label{doebl}
 		n(2s_*,s,x)=P_{2s_*}n_0(s,x)\ge p_{*}e^{-2p_{\infty}s_*}\,\mathds{1}_{[0,s_*]}(s)\,g(x)\quad\forall(s,x)\in(0,\infty)\times\Omega.
 	\end{equation}
	This means that the semi-group $P_t$ associated to equation \eqref{P_t} satisfies Doeblin's condition with $t_0=2s_*,\,\alpha=p_{*}s_*e^{-2p_{\infty}s_*}$ and $\nu=\tfrac{1}{s_*}\mathds{1}_{[0,s_*]}(s)$ for functions $n_0(\cdot,x)\in L^1_s$ with $g(x)=1$.
	
	Let $x\in\Omega$ be fixed and consider $\tilde{P}_t\colon L^1_{s}\to L^1_{s}$ the semi-group associated with the problem
	\begin{equation*}
	\left\{
	\begin{matrix*}[l]
	\partial_t \tilde{n}+\partial_s \tilde{n}+p(s,S(x))\tilde{n}=0&t>0,s>0\vspace{0.15cm},\\
	\tilde{n}(t,s=0,x)=0&t>0\vspace{0.15cm},\\
	\tilde{n}(t=0,s,x)=n_0(s,x)&s\ge0.
	\end{matrix*}
	\right.
	\end{equation*}
	In this case the solution is given by
	\begin{equation}
	\tilde{P}_t n_0(s,x)=n_0(s-t,x)\exp\left(-\int_0^tp(s-t+\tau,S(x))\,d\tau\right)\mathds{1}_{\{ s>t \}}.
	\end{equation}
	Then the solution of \eqref{P_t} satisfies
	$$n(t,s,x)=\tilde{P}_t n_0(s,x)+\int_0^t\tilde{P}_{t-\tau}(N(\tau,x)\delta_0(s))\,d\tau.$$
	Moreover we have the following inequalities
	\begin{equation*}
	\begin{matrix}
	n(t,s,x)\ge \tilde{P}_t n_0(s,x)=n_0(s-t,x)\exp\left(-\int_0^tp(s-t+\tau,S(x))\,d\tau\right)\ge n_0(s-t,x)e^{-p_{\infty}t}\mathds{1}_{\{ s>t\}}.\vspace{0.15cm}\\
	\tilde{P}_{t-\tau}n_0(s,x)\ge n_0(s-t+\tau,x)e^{-p_{\infty}(t-\tau)}\mathds{1}_{\{ s>t-\tau\}}.
	\end{matrix}
	\end{equation*}
	Then for $t>s_*$ we get
	\begin{equation*}
	\begin{split}
	N(t,x)&=\int_0^\infty p(s,S(x))n(t,s,x)\,ds\\
	&\ge p_{*}\int_{s_*}^\infty n(t,s,x)\,ds\\
	&\ge p_{*}\int_t^\infty n(t,s,x)\,ds\\
	&\ge p_{*}e^{-p_{\infty}t}\int_t^\infty n_0(s-t,x)\,ds\\
	&\ge p_{*}e^{-p_{\infty}t}g(x).
	\end{split}
	\end{equation*}
	In that case for any $s>0$ and $t>s+s_*$ we have that
	\begin{equation*}
	\begin{split}
	n(t,s,x)&\ge\int_0^t\tilde{P}_{t-\tau}(N(\tau,x)\delta_0(s))\,d\tau\\
	&\ge\int_{s_*}^t\tilde{P}_{t-\tau}(p_{*}e^{-p_{\infty}\tau}g(x)\delta_0(s))\,d\tau\\
	&\ge p_{*}\int_{s_*}^t\delta_0(s-t+\tau)e^{-p_{\infty}\tau}e^{-p_{\infty}(t-\tau)}g(x)\mathds{1}_{\{ s-t+\tau>0\}}\,d\tau\\
	&\ge p_{*}e^{-p_{\infty}t}\mathds{1}_{\{0<s<t-s_*\}}\,g(x).
	\end{split}
	\end{equation*}
	Therefore we get the estimate \eqref{doebl} by choosing $t=2s_*$.  Finally, the exponential convergence to equilibrium readily follows from Doeblin's theorem with $\lambda=-\frac{\ln(1-\alpha)}{t_0}>0$ and from normalizing by $g(x)$.
\end{proof}

\begin{rmk}
	Doeblin's condition is also verified for the case $p$ defined in \eqref{pjump}, even when $\sigma$ is unbounded. Since the amplitude $S$ is uniformly bounded in the system \eqref{eq0}, we can relax the condition \eqref{lbp2} for $S$ lying in some bounded interval instead of for all $S\in\R$. Therefore the exponential convergence to equilibrium is valid as well.
\end{rmk}

\subsubsection{The non-linear case}
The linear theory allows to determine the asymptotic behavior of the non-linear system \eqref{eq0} for the weak interconnection regime as well. By using Duhamel's formula, it is possible to conclude the improved version of theorem \ref{conveq0}.
\begin{thm}[Improved convergence to equilibrium]
	\label{conveq1}
	Assume \eqref{inidata}-\eqref{gx} and that $p\in W^{1,\infty}((0,\infty)\times\R)$ satisfies \eqref{lbp2}. For $\gamma$ and $\|\tfrac{\partial p}{\partial S}\|_\infty$ small enough let $(n^*,N^*,S^*,w^*)$ be the corresponding stationary state of \eqref{eq0}. Then there exist $C,\lambda>0$ such that the solution $n$ of \eqref{eq0} satisfies
	$$\|n(t)-n^*\|_{L^\infty_x L^1_s}+\|w(t)-w^*\|_\infty\le Ce^{-\lambda t}\left(\|n_0-n^*\|_{L^\infty_x L^1_s}+\|w_0-w^*\|_\infty\right),\:\forall t\ge0.$$
	Moreover $\|S(t)-S^*\|_\infty$ and $\|N(t)-N^*\|_\infty$ converge exponentially to $0$ when $t\to\infty$.
\end{thm}

\begin{proof}
	Observe that $n$ satisfies the evolution equation
	$$\partial_t n=\mathcal{L}_S [n]\coloneqq-\partial_s n-p(s,S)n+\delta_0(s)\int_0^\infty p(u,S(t,x))n(t,u,x)\,du.$$
	We can rewrite the evolution equation as
	\begin{equation}
	\label{nt=Ln}
	\partial_t n=\mathcal{L}_{S^*}[n]+(\mathcal{L}_{S}[n]-\mathcal{L}_{S^*}[n])=\mathcal{L}_{S^*}[n]+h.
	\end{equation}
	with $h(t,s,x)$ given by
	\begin{equation}
	h=\big(p(s,S^*(x))-p(s,S(t,x))\big)n(t,s,x)+\delta_0(s)\int_0^\infty\big(p(u,S(t,x))-p(u,S^*(x))\big)n(t,u,x)\,du.
	\end{equation}
	Let $P_t$ be the linear semi-group associated to operator $\mathcal{L}_{S^*}$. Since $P_t n^*=n^*$ for all $t\ge0$, we get that $n$ satisfies
	\begin{equation}
	\label{n-n*}
	n-n^*=P_t (n_0-n^*)+\int_0^t P_{t-\tau}h(\tau,s,x)\,d\tau,
	\end{equation}
	so we need find an estimate for the function $h$. Analogously to the proof of theorem \ref{conveq0}, we have the following inequalities:
	\begin{equation*}
		\begin{matrix}
		\|S(t)-S^*\|_\infty\le p_\infty\|w(t)-w^*\|_\infty+\gamma|\Omega|\,\|N(t)-N^*\|_\infty,\vspace{0.15cm}\\
		\|N(t)-N^*\|_\infty\le\|\tfrac{\partial p}{\partial S}\|_\infty\|S(t)-S^*\|_\infty+p_\infty\|n(t)-n^*\|_{L^\infty_xL^1_s},
		\end{matrix}
	\end{equation*}
    With $C_1\coloneqq\gamma|\Omega|\|\tfrac{\partial p}{\partial S}\|_\infty<1$, we get from these inequalities
	\begin{equation}
	\label{dSN}
	\begin{matrix}
	\displaystyle\|S(t)-S^*\|_\infty\le \frac{p_\infty}{1-C_1}\left(\|w(t)-w^*\|_\infty+\gamma|\Omega|\,\|n(t)-n^*\|_{L^\infty_xL^1_s}\right),\vspace{0.15cm}\\
	\displaystyle \|N(t)-N^*\|_\infty\le\frac{p_\infty}{1-C_1}\left(\|\tfrac{\partial p}{\partial S}\|_\infty\|w(t)-w^*\|_\infty+\|n(t)-n^*\|_{L^\infty_x L^1_s}\right).
	\end{matrix}
	\end{equation}
	Thus for $h$ we get
	\begin{equation}
	\begin{split}
	\|h(t)\|_{L^\infty_xL^1_s}&\le 2\|g\|_\infty\|\tfrac{\partial p}{\partial S}\|_\infty\|S(t)-S_*\|_\infty\\
	&\le\frac{2p_\infty\|g\|_\infty\|\tfrac{\partial p}{\partial S}\|_\infty}{1-C_1}\left(\|w(t)-w^*\|_\infty+\gamma|\Omega|\,\|n(t)-n^*\|_{L^\infty_xL^1_s}\right)\\
	&\le C_2\left(\|w(t)-w^*\|_\infty+\|n(t)-n^*\|_{L^\infty_xL^1_s}\right),
	\end{split}
	\end{equation}
	with $C_2\coloneqq\frac{2p_\infty\|g\|_\infty\|\tfrac{\partial p}{\partial S}\|_\infty}{1-C_1}\max\{1,\gamma|\Omega|\,\}$. On the one hand, using theorem \ref{DoeblinConv} and the fact that $\int_0^\infty h(t,s,x)\,ds=0$, we get from \eqref{n-n*}
	\begin{equation*}
	\begin{split}
		\|n(t)-n_*\|_{L^\infty_xL^1_s}&\le \|P_t(n_0-n_*)\|_{L^\infty_xL^1_s}+\int_0^t\|P_{t-\tau}h(\tau)\|_{L^\infty_xL^1_s}\,d\tau\\
		&\le \frac{e^{-\lambda t}}{1-\alpha}\|n_0-n_*\|_{L^\infty_xL^1_s}+\frac{1}{1-\alpha}\int_0^te^{-\lambda(t-\tau)}\|h(\tau)\|_{L^\infty_xL^1_s}\,d\tau\\
		&\le\frac{e^{-\lambda t}}{1-\alpha}\|n_0-n_*\|_{L^\infty_xL^1_s}+\frac{C_2}{1-\alpha}\int_0^t e^{-\lambda(t-\tau)}\left(\|w(\tau)-w^*\|_\infty+\|n(\tau)-n^*\|_{L^\infty_xL^1_s}\right)\,d\tau,
	\end{split}
	\end{equation*}
	with $\alpha=p_{*}s_*e^{-2p_{\infty}s_*},\,\lambda=-\tfrac{\ln(1-\alpha)}{2s_*}>0$. On the other hand, from the second inequality in \eqref{dSN} we deduce
	\begin{equation*}
	\begin{split}
	\|w(t)-w^*\|_\infty&\le e^{-t}\|w_0-w^*\|_\infty+2\gamma\,\int_0^t e^{-(t-\tau)}\|N(\tau)-N^*\|_\infty\,d\tau\\
	&\le e^{-t}\|w_0-w^*\|_\infty+C_3\int_0^t e^{-(t-\tau)}\left(\|w(\tau)-w^*\|_\infty+\|n(\tau)-n^*\|_{L^\infty_xL^1_s}\right)\,d\tau,
	\end{split}		
	\end{equation*}
	with $C_3\coloneqq\frac{2\gamma p_\infty}{1-C_1}\max\{\|\tfrac{\partial p}{\partial S}\|_\infty,1\}$. Hence we get
	\begin{equation*}
		\begin{split}
	\|n(t)-n_*\|_{L^\infty_xL^1_s}+\|w(t)-w^*\|_\infty&\le \frac{e^{-\tilde\lambda t}}{1-\alpha}\left(\|n_0-n_*\|_{L^\infty_xL^1_s}+\|w_0-w^*\|_\infty\right)\\
	&\quad+C_4 e^{-\tilde\lambda t}\int_0^t e^{\tilde\lambda\tau}\left(\|w(\tau)-w^*\|_\infty+\|n(\tau)-n^*\|_{L^\infty_xL^1_s}\right)\,d\tau,
		\end{split}
	\end{equation*}
	with $\tilde\lambda\coloneqq\min\{\lambda,1\},\,C_4\coloneqq\max\left\{\frac{C_2}{1-\alpha},C_3\right\}$. Therefore, by using Gronwall's inequality we have
	$$	\|n(t)-n_*\|_{L^\infty_xL^1_s}+\|w(t)-w^*\|_\infty\le\frac{e^{-(\tilde\lambda-C_4)t}}{1-\alpha}\left(\|n_0-n_*\|_{L^\infty_xL^1_s}+\|w_0-w^*\|_\infty\right).$$
	So we get the result if $\gamma$ and $\|\tfrac{\partial p}{\partial S}\|_\infty$ are small enough so that $C_4<\tilde{\lambda}$. The exponential convergence of $N$ and $S$ readily follows from the estimates in \eqref{dSN}.
\end{proof}

\begin{rmk}
	If in addition $n_0\in L^\infty_{s,x}$, the result is also valid for $p$ defined in \eqref{pjump} by replacing the estimates involving $\|\tfrac{\partial p}{\partial S}\|_\infty$ by its equivalent with $\|\sigma'\|_\infty$ small enough.
\end{rmk}

\subsection{Effect of large inputs}
We now study the asymptotic behavior for a large enough input in the system \eqref{eq0}. For $k>0$ consider $n^k(t,s,x)$ a solution of the system 
\begin{equation}
\label{eqbeta}
\left\{
\begin{matrix*}[l]
\partial_t n+\partial_s n+p(s,S(t,x))n=0&t>0,s>0,x\in\Omega\vspace{0.15cm},\\
N(t,x)\coloneqq n(t,s=0,x)=\int_0^\infty p(s,S(t,x))n\,ds&t>0,x\in\Omega\vspace{0.15cm},\\
S(t,x)=\int_\Omega w(t,x,y)N(t,y)dy+k I(x)&t>0,x\in\Omega\vspace{0.15cm},\\
\partial_t w= -w + \gamma G(N(t,x),N(t,y))&t>0,\,x,y\in\Omega\vspace{0.15cm},\\
n(t=0,s,x)=n_0(s,x)\ge0,\:w(t=0,x,y)=w_0(x,y)\ge0&s\ge0,\,x,y\in\Omega.
\end{matrix*}
\right.
\end{equation}
We prove by the means of Deoblin's theroy that if $k$ tends to infinity, then the solutions of \eqref{eqbeta} converge to a solution of linear problem \eqref{eql}.
\begin{thm}
	\label{largeinput}
	Assume \eqref{inidata}-\eqref{gx} with $p\in W^{1,\infty}((0,\infty)\times\R)$ satisfying \eqref{lbp2} and such that $p(s,\infty)\coloneqq\lim_{S\to\infty} p(s,S)$ exists for all $s\ge0$. Moreover suppose that $I(x)>0$ almost everywhere in $\Omega$. Let $n^\infty$ be the solution of linear problem
	\begin{equation}
	\left\{
	\begin{matrix*}[l]
	\partial_t n+\partial_s n+p(s,\infty)n=0&t>0,s>0,x\in\Omega\vspace{0.15cm},\\
	N(t,x)\coloneqq n(t,s=0,x)=\int_0^\infty p(s,\infty)n\,ds&t>0,x\in\Omega\vspace{0.15cm},\\
	n(t=0,s,x)=n_0(s,x)&s\ge0,x\in\Omega.
	\end{matrix*}
	\right.
	\end{equation}
	Then for all $t>0$ we have $n^k(t)\to n^\infty(t)$ in $L^1_{s,x}$ when $k\to\infty$.
\end{thm}
\begin{proof}
	Let $\mathcal{L}_S$ be the operator defined in \eqref{nt=Ln}. In the same way we define the operator $\mathcal{L}_\infty$ given by
	$$\mathcal{L}_\infty[n]\coloneqq-\partial_s n-p(s,\infty)n+\delta_0(s)\int_0^\infty p(u,\infty)n(t,u,x)\,du.$$
	Thus we rewrite the evolution equation of $n^k$ as
	\begin{equation*}
	\partial_t n^k=\mathcal{L}_{\infty}[n]+(\mathcal{L}_{S}[n]-\mathcal{L}_{\infty}[n])=\mathcal{L}_{\infty}[n]+h.
	\end{equation*}
	with $h(t,s,x)$ given by
	\begin{equation*}
	h=(p(s,\infty)-p(s,S(t,x)))n(t,s,x)+\delta_0(s)\int_0^\infty(p(u,S(t,x))-p(u,\infty))n(t,u,x)\,du,
	\end{equation*}
	so we get 
	$$\|h\|_{L^1_{s,x}}\le 2\iint_0^\infty|p(s,\infty)-p(s,S(t,x))|n(t,s,x)\,ds\,dx.$$
	Since $S(t,x)\ge k I(x)$ we get that for all $t>0$ and a.e. $x\in\Omega$ that $S(t,x)\to\infty$ when $k\to\infty$ and thus for all $s\ge0$ we have $p(s,S(t,x))\to p(s,\infty)$. From the method of characteristics we get that $n$ satisfies
	$$n^k(t,s,x)\le n_0(t-s,x)+p_\infty g(x)\mathds{1}_{\{0<s<t\}},$$
	hence by Lebesgue's theorem we conclude for all $t>0$ that $\|h(t)\|_{L^1_{s,x}}\to0$ when $k\to\infty$.
	
	Let $P_t$ be the semi-group associated to $\mathcal{L}_\infty$. Since $P_t[n_0]=n^\infty$ we get that $n^k$ satisfies
	$$n^k-n^\infty=\int_0^t P_{t-\tau}h(\tau,s,x)\,d\tau.$$
	Since $\int_0^\infty h(t,s,x)\,ds=0$ we conclude by Doeblin's theorem that
	\begin{equation*}
	\begin{split}
	\|n^k(t)-n^\infty(t)\|&\le\int_0^t\|P_{t-\tau}h(\tau)\|_{L^1_{s,x}}\,d\tau\\
	&\le\int_0^t e^{-(t-\tau)}\|h(\tau)\|_{L^1_{s,x}}\,d\tau.
	\end{split}
	\end{equation*}
	And since $\|h(t)\|_{L^1_{s,x}}\le 4p_\infty$, we conclude the result by Lebesgue's theorem.
\end{proof}

\begin{rmk}
	In the case of $p$ defined in \eqref{pjump} the same result holds if $\lim_{S\to\infty}\sigma(S)$ exists. Moreover for the particular case $\sigma(S)=S$, the result is straightforward from the fact that $p(s,\infty)=0$ and $N^k\to 0$ so $n^\infty$ is solution of a simple transport equation.
\end{rmk}
\section{Slow learning dynamics}
\label{slowlearning}
From a neuroscience viewpoint we can assume that the learning dynamics are much slower than the elapsed time dynamics. This is represented by the rescaled system
\begin{equation}
\label{eqeps}
\left\{
\begin{matrix*}[l]
\varepsilon\partial_t n+\partial_s n+p(s,S(t,x))n=0&t>0,s>0,x\in\Omega\vspace{0.15cm},\\
N(t,x)\coloneqq n(t,s=0,x)=\int_0^\infty p(s,S(t,x))n\,ds&t>0,x\in\Omega\vspace{0.15cm},\\
S(t,x)=\int w(t,x,y)N(t,y)dy+I(t,x)&t>0,x\in\Omega\vspace{0.15cm},\\
\partial_t w= -w+ \gamma G(N(t,x),N(t,y))&t>0,\,x,y\in\Omega\vspace{0.15cm},\\
n(t=0,s,x)=n_0(s,x)\ge0,\:w(t=0,x,y)=w_0(x,y)\ge0&s\ge0\,x,y\in\Omega,
\end{matrix*}
\right.
\end{equation}
with $\varepsilon>0$ small enough. This means that the time scale for $w$ is of order $1$, while $n$ relaxes very quickly to equilibrium with time scale $\varepsilon$. Well-posedness and exponential convergence results are also valid for this system. 

Let $n^\varepsilon(t,s,x)$ be the solution of \eqref{eqeps}, we are interested in the asymptotic behavior of $n^\varepsilon$ when $\varepsilon\to0$. In order to do so, consider the formal limit system which corresponds to take $\varepsilon=0$ in \eqref{eqeps}
\begin{equation}
\label{eqeps0}
\left\{
\begin{matrix*}[l]
\partial_s n+p(s,S(t,x))n=0&s>0,x\in\Omega\vspace{0.15cm},\\
N(t,x)\coloneqq n(t,s=0,x)=\int_0^\infty p(s,S(t,x))n\,ds& t>0,x\in\Omega\vspace{0.15cm},\\
S(t,x)=\int w(t,x,y)N(t,y)dy+I(t,x)&t>0,x\in\Omega\vspace{0.15cm},\\
\partial_t w= -w + \gamma G(N(t,x),N(t,y))&t>0,\,x,y\in\Omega\vspace{0.15cm},\\
w(t=0,x,y)=w_0(x,y)& x,y\in\Omega.
\end{matrix*}
\right.
\end{equation}
The question here is to determine if $n^\varepsilon$, the solution of system \eqref{eqeps}, converges to some solution of \eqref{eqeps0} when $\varepsilon$ vanishes. In order to address this question, we first prove that problem \eqref{eqeps0} is well-posed under the weak interconnection regime.
\begin{thm}[Existence for system \eqref{eqeps0}]
\label{exist0}
Consider $g\in \C_b(\Omega)$ and $F$ be the function defined in \eqref{CapF}. Then under the condition $$\max\left\{\|w_0\|_\infty,\gamma\right\}\|F'\|_\infty<1$$
the system \eqref{eqeps0} has a unique solution satisfying $\int_0^\infty n(t,s,x)\,ds=g(x)$ for all $t\ge0$.
\end{thm}
To prove the result we need the following lemma.
\begin{lem}
	Consider $w\in\C_b([0,\infty)\times\Omega\times\Omega)$ fixed. Then the operator $\mathcal{T}\colon\C_b([0,\infty)\times\Omega)\to\C_b([0,\infty)\times\Omega)$ defined by
	$$\mathcal{T}[S](t,x)=\int w(t,x,y)g(y)F(S(t,y))\,dy+I(t,x),$$
	has a unique fixed point $\bar{S}\in\C_b([0,\infty)$ if $\|w\|_\infty\|F'\|_\infty<1$. Moreover, $\bar{S}$ is a locally-Lipschitz function of $w$.
\end{lem}
\begin{proof}
	We first notice that $\mathcal{T}$ is a contraction. In fact for $S_1,S_2\in\C_b([0,\infty)\times\Omega)$ we have
	$$\|\mathcal{T}[S_1]-\mathcal{T}[S_2]\|_\infty\le\|w\|_\infty\|F'\|_\infty\|S_1-S_2\|_\infty.$$
	Hence by Picard's theorem there is a unique fixed point $\bar{S}[w]\in\C_b([0,\infty)\times\Omega)$.
	
	Now consider $\bar{S}[w_1],\bar{S}[w_2]$ the respective fixed points associated to $w_1,w_2$. Then we have the following estimate
	$$\|\bar{S}[w_1]-\bar{S}[w_2]\|_\infty\le \|F\|_\infty\|w_1-w_2\|_\infty+\|w_2\|_\infty\|F'\|_\infty\|\bar{S}[w_1]-\bar{S}[w_2]\|_\infty$$
	and hence
	$$\|\bar{S}[w_1]-\bar{S}[w_2]\|_\infty\le\frac{\|F\|_\infty}{1-\|w_2\|_\infty\|F'\|_\infty}\|w_1-w_2\|_\infty$$
	so $\bar{S}$ is a locally Lipschitz function of $w$.
\end{proof}

In this setting, we continue with the proof of theorem \ref{exist0}.

\begin{proof}
	First observe that $n$ satisfies
	$$n(t,s,x)=N(t,x)e^{-\int_0^sp(\tau,S(t,x))\,d\tau}.$$
	and by integrating with respect to $s$, we get the following expression for $N$
	$$N(t,x)=g(x)\left(\int_0^\infty e^{-\int_0^sp(\tau,S(t,x))\,d\tau}\,ds\right)^{-1}=g(x)F(S(t,x)).$$
	Hence the problem is reduced to the following system for $(S,w)$
	\begin{equation}
	\label{limwS}
	\left\{
	\begin{matrix*}[l]
	S(t,x)=\int w(t,x,y)g(y)F(S(t,y))dy+I(t,x)&t>0,x\in\Omega\vspace{0.15cm},\\	
	\partial_t w=-w+\gamma G\left(g(x)F(S(t,x)),g(y)F(S(t,y))\right)&t>0,\,x,y\in\Omega\vspace{0.15cm},\\
	w(t=0,x,y)=w_0(x,y)&x,y\in\Omega.
	\end{matrix*}
	\right.
	\end{equation}
	Since we have a uniform estimate for $w$ in \eqref{bdw}, we conclude that $\bar{S}[w]$ is a Lipschitz function restricted to the set $$U=\left\{ w\in\C_b([0,\infty)\times\Omega\times\Omega)\colon\|w\|_\infty\le\max\{\|w_0\|_\infty,\gamma\right\}\}$$ 
	if $\max\{\|w_0\|_\infty,\gamma\}\|F'\|_\infty<1$. So by applying the Cauchy-Lipschitz-Picard theorem, we conclude that system \eqref{limwS} has a unique solution, defined in some time interval $[0,T]$. Finally, by noting again that $w$ is uniformly bounded as in \eqref{bdw}, we can iterate this argument to get a solution globally defined in time.
\end{proof}

By replicating the proof in theorem \ref{conveq0}, we get for system \eqref{eqeps0} its asymptotic behavior when $t\to\infty$.
\begin{thm}[Long term behavior for system \eqref{eqeps0}]
\label{slowth}	
Assume \eqref{inidata}-\eqref{gx} and that $p\in W^{1,\infty}((0,\infty)\times\R)$ satisfies \eqref{lbp1}. For $\gamma$ and $\|\frac{\partial p}{\partial S}\|_\infty$ small enough, consider $(n^*,N^*,S^*,w^*)$ the corresponding stationary state of \eqref{eq0}. Then there exist $C,\lambda>0$ such that the solution of \eqref{eqeps0} satisfies
\begin{equation}
	\|n(t)-n^*\|_{L^1_{s,x}}+\|w(t)-w^*\|_{L^1_{x,y}}\le Ce^{-\lambda t}\|w_0-w^*\|_{L^1_{x,y}},\:\forall t\ge 0.
\end{equation}
Moreover $\|S(t)-S^*\|_{L^1_x}$ and $\|N(t)-N^*\|_{L^1_x}$ converge exponentially to $0$ when $t\to\infty$.
\end{thm}

Next we prove the convergence of $n^\varepsilon$ for the case of weak interconnection when the firing rate is strictly positive, by means of the entropy method.
\begin{thm}[Convergence for \eqref{eqeps} as $\varepsilon\to0$]
\label{conveps}
Assume \eqref{inidata}-\eqref{gx} with $n_0\in W^{1,1}_{s,x}$ and that $p\in W^{1,\infty}((0,\infty)\times\R)$ satisfies \eqref{lbp1}. For $\max\{\|w_0\|_\infty,\gamma\}$ small enough, let $(n^\varepsilon,N^\varepsilon,S^\varepsilon,w^\varepsilon)$ be the solution of system \eqref{eqeps} and let $(\bar{n},\bar{N},\bar{S},\bar{w})$ be the unique solution of system \eqref{eqeps0}.

Then for all $T>0$ we have $n^\varepsilon\to\bar{n}$ in $L^1((0,T)\times(0,\infty)\times\Omega)$ and $w^\varepsilon\to\bar{w}$ in $L^1((0,T)\times\Omega\times\Omega)$. Moreover $N^\varepsilon\to\bar{N}$ and $S^\varepsilon\to\bar{S}$ in $L^1((0,T)\times\Omega)$.
\end{thm}

\begin{proof}
	Let $(n^\varepsilon,N^\varepsilon,S^\varepsilon,w^\varepsilon)$ be the solution of system \eqref{eqeps}. We start by reminding the following uniform estimates
	\begin{equation}
	\label{unifnsw}
		\|N^\varepsilon(t)\|_\infty\le p_{\infty}\|g\|_\infty,\quad
		\|w^\varepsilon\|_\infty\le\max\{\|w_0\|_\infty,\gamma\},\quad\forall t\ge0,\,\varepsilon>0,
	\end{equation}.
	
	The first step is to estimate $u=\partial_tn^\varepsilon$, which satisfies the following equation
	$$\varepsilon\partial_t u+\partial_s u+p(s,S^\varepsilon)u+\frac{\partial p}{\partial S}(s,S^\varepsilon)\,\partial_tS^\varepsilon \,n^\varepsilon=0,$$
	thus we have the following inequality
	$$\varepsilon\partial_t|u|+\partial_s |u|+p(s,S^\varepsilon)|u|\le\left\|\frac{\partial p}{\partial S}\right\|_\infty|\partial_tS^\varepsilon|\,n^\varepsilon.$$
	By integrating with respect to all variables, we get
	\begin{equation}
	\label{intpu}
	\begin{split}
	\int_0^T\hspace{-0.2cm}\iint_0^\infty p(s,S^\varepsilon)|u|\,ds\,dx\,dt&	\le\varepsilon\iint_0^\infty|u|(0,s,x)\,ds\,dx+\int_0^T\hspace{-0.2cm}\int |\partial_tN^\varepsilon|(t,x)\,dx\,dt\\
	&\quad+\left\|\tfrac{\partial p}{\partial S}\right\|_\infty\int_0^T\|\partial_tS^\varepsilon(t,\cdot)\|_\infty\,dt.
	\end{split}	
	\end{equation}
	Thus we have to estimate each term in the right-hand side. For the first it readily follows that
	\begin{equation}
	\label{nt0}
		\varepsilon\iint_0^\infty|u|(0,s,x)\,ds\,dx\le\iint_0^\infty|\partial_sn_0|\,ds\,dx+\iint_0^\infty p(s,S^\varepsilon(0,x))n_0\,ds\,dx\le \|n_0\|_{W^{1,1}_{s,x}}+p_{\infty}.
	\end{equation}
	Next, for $\partial_tN^\varepsilon$ we have
	$$\partial_tN^\varepsilon(t,x)=\partial_tS^\varepsilon(t,x)\int_0^\infty\frac{\partial p}{\partial S}(s,S^\varepsilon(t,x))\,n^ \varepsilon(t,s,x)\,ds+\int_0^\infty p(s,S^\varepsilon(t,x))\,\partial_tn^\varepsilon(t,s,x)\,ds.$$
	Thus for the second term we get
	\begin{equation}
	\label{dtN}
		\int_0^T\hspace{-0.2cm}\int |\partial_tN^\varepsilon|(t,x)\,dx\,dt\le\left\|\tfrac{\partial p}{\partial S}\right\|_\infty\int_0^T\|\partial_tS^\varepsilon(t,\cdot)\|_\infty\,dt+\int_0^T\hspace{-0.2cm}\int\left|\int_0^\infty p(s,S^\varepsilon)u\,ds\right|dx\,dt.
	\end{equation}
	On the other hand, for $\partial_tS^\varepsilon$ we get
	$$\partial_t S^\varepsilon(t,x)=\int\partial_tw^\varepsilon(t,x,y) N^\varepsilon(t,y)\,dy+\int w^\varepsilon(t,x,y) \partial_t N^\varepsilon(t,y)\,dy+\partial_t I(t,x).$$
	Let $A\coloneqq\max\{\|w_0\|_\infty,\gamma\}$, by using the uniform estimates in \eqref{unifnsw} we obtain
	\begin{equation}
	\label{dtS}
	\begin{split}
	\int_0^T\|\partial_t S^\varepsilon(t,\cdot)\|_\infty\,dt&\le \|\partial_t w^\varepsilon\|_\infty\int_0^T\hspace{-0.2cm}\int N^\varepsilon\,dy\,dt+\|w^\varepsilon\|_\infty\int_0^T\hspace{-0.2cm}\int|\partial_t N^\varepsilon|\,dy\,dt+\|\partial_tI\|_\infty T\\
	&\le A\left(2p_{\infty}T+ \int_0^T\hspace{-0.2cm}\int|\partial_t N^\varepsilon|\,dy\,dt\right)+\|\partial_t I\|_\infty T.
	\end{split}
	\end{equation}
	Let $\beta_1\coloneqq\|\frac{\partial p}{\partial S}\|_\infty\max\{\|w_0\|_\infty,\gamma\}<1$. Hence from \eqref{dtN} we conclude
	\begin{equation}
	\int_0^T\|\partial_t S^\varepsilon(t,\cdot)\|_\infty\,dt\le\frac{1}{1-\beta_1}\left(2A p_{\infty}T+\|\partial_t I\|_\infty T+Ap_{\infty}\int_0^T\hspace{-0.2cm}\iint_0^\infty |u|\,ds\,dx\,dt\right).
	\end{equation}
	Therefore we can deduce from \eqref{intpu} the following estimate
	\begin{equation}
	\begin{split}
	\int_0^T\hspace{-0.2cm}\iint_0^\infty p(s,S^\varepsilon)|u|\,ds\,dx\,dt&\le\|n_0\|_{W^{1,1}_{s,x}}+p_{\infty}+\frac{2\|\frac{\partial p}{\partial S}\|_\infty T}{1-\beta_1}\left(2A p_{\infty}+\|\partial_t I\|_\infty\right)\\
	&\quad+\frac{2\beta_1}{1-\beta_1}p_{\infty}\int_0^T\hspace{-0.2cm}\iint_0^\infty |u|\,ds\,dx\,dt+\int_0^T\hspace{-0.2cm}\int\left|\int_0^\infty p(s,S^\varepsilon)u\,ds\right|dx\,dt.
	\end{split}
	\end{equation}
	At this stage we can use again the entropy trick from \cite{michel2005general,perthame2006transport}. Since $\int_0^\infty u\,ds=0$ and $p\ge p_{*}$, we have the following inequality
	$$\int_0^T\hspace{-0.2cm}\int\left|\int_0^\infty p(s,S^\varepsilon)u\,ds\right|dx\,dt=\int_0^T\hspace{-0.2cm}\int\left|\int_0^\infty (p(s,S^\varepsilon)-p_{*})u\,ds\right|dx\,dt\le\int_0^T\hspace{-0.2cm}\int\hspace{-0.2cm}\int_0^\infty(p(s,S^\varepsilon)-p_{*})|u|\,ds\,dx\,dt.$$
	As $\beta_1$ is small enough, we conclude the $L^1$ norm of $u=\partial_tn^\varepsilon$ is uniformly bounded in $\varepsilon$.
	\begin{equation}
	\begin{split}
	\left(p_{*}-\tfrac{2\beta_1}{1-\beta_1}p_{\infty}\right)\int_0^T\hspace{-0.2cm}\iint_0^\infty |\partial_tn^\varepsilon|\,ds\,dx\,dt&\le\frac{2\|\frac{\partial p}{\partial S}\|_\infty T}{1-\beta_1}\left(2Ap_{\infty}+\|\partial_t I\|_\infty\right)\\
	&\quad+\|n_0\|_{W^{1,1}_{s,x}}+p_{\infty}.
	\end{split}
	\end{equation}
	The next step is to estimate $n^\varepsilon-\bar{n}$, by using a similar argument. Let $\bar{N},\bar{S}$ and $\bar{w}$ be the terms associated to $\bar{n}$ in the system \eqref{eqeps0}, so that we have
	$$\partial_s(n^\varepsilon-\bar{n})+p(s,S^\varepsilon)(n^\varepsilon-\bar{n})=-\varepsilon\partial_tn^\varepsilon-(p(s,S^\varepsilon)-p(s,\bar{S}))\bar{n}.$$
	Hence we have the following inequality
	$$\partial_s|n^\varepsilon-\bar{n}|+p(s,S^\varepsilon)|n^\varepsilon-\bar{n}|\le\varepsilon |\partial_tn^\varepsilon|+\left\|\tfrac{\partial p}{\partial S}\right\|_\infty|S^\varepsilon-\bar{S}|\,\bar{n}.$$
	By integrating with respect to all variables we get
	\begin{equation}
	\begin{split}
	\label{n-n0}
	\int_0^T\hspace{-0.2cm}\iint_0^\infty p(s,S^\varepsilon)|n^\varepsilon-\bar{n}|\,ds\,dx\,dt&\le\varepsilon\int_0^T\hspace{-0.2cm}\iint_0^\infty|\partial_tn^\varepsilon|\,ds\,dx\,dt+\int_0^T\hspace{-0.20cm}\int|N^\varepsilon-\bar{N}|\,dx\,dt\\
	&\quad+\|g\|_\infty \left\|\tfrac{\partial p}{\partial S}\right\|_\infty \int_0^T\hspace{-0.20cm}\int|S^\varepsilon-\bar{S}|\,dx\,dt.
	\end{split}	
	\end{equation}
	So we have to estimate the respective terms involving $N$ and $S$. For $N^\varepsilon-\bar{N}$ we have
	\begin{equation}
	\label{N-N0}
	\int_0^T\hspace{-0.2cm}\int|N^\varepsilon-\bar{N}|\,dx\,dt\le\|g\|_\infty \left\|\tfrac{\partial p}{\partial S}\right\|_\infty\int_0^T\hspace{-0.2cm}\int|S^\varepsilon-\bar{S}|\,dx\,dt+\int_0^T\hspace{-0.2cm}\int\left|\int_0^\infty p(s,S^\varepsilon)(n^\varepsilon-\bar{n})\,ds\right|dx\,dt.
	\end{equation}
	In order to  estimate $S^\varepsilon-\bar{S}$, we need to estimate $w^\varepsilon-\bar{w}$ first. By using formula \eqref{solw} we obtain
	\begin{equation*}
	\begin{split}
		\int_0^T\hspace{-0.2cm}\iint|w-\bar{w}|\,dx\,dy\,dt&\le2\gamma|\Omega|\,\int_0^T\int_0^te^{-(t-\tau)}\|N^\varepsilon-\bar{N}\|_{L^1_x}(\tau)\,d\tau\,dt\\
		&\le 2\gamma|\Omega|\,\int_0^T\int_\tau^te^{-(t-\tau)}\|N^\varepsilon-\bar{N}\|_{L^1_x}(\tau)\,dt\,d\tau\\
		&\le 2\gamma|\Omega|\,\int_0^T\|N^\varepsilon-\bar{N}\|_{L^1_x}(\tau)\left(\int_\tau^Te^{-(t-\tau)}\,dt\right)d\tau,
	\end{split}
	\end{equation*}
	so we conclude the following estimate
	\begin{equation}
	\label{w-w0}
	\int_0^T\hspace{-0.2cm}\iint|w-\bar{w}|\,dx\,dy\,dt\le2\gamma|\Omega|\,\int_0^T\hspace{-0.2cm}\int|N^\varepsilon-\bar{N}|\,dx\,dt.
	\end{equation}
	Thus, for $S^\varepsilon-\bar{S}$ we get
	\begin{equation*}
	\begin{split}
	\int_0^T\hspace{-0.2cm}\int|S^\varepsilon-\bar{S}|\,dx\,dt&\le\int_0^T\hspace{-0.2cm}\iint|w-\bar{w}| N^\varepsilon(t,y)\,dx\,dy\,dt+|\Omega|\,\|\bar{w}\|_\infty\int_0^T\hspace{-0.2cm}\int|N^\varepsilon-\bar{N}|\,dy\,dt\\
	&\le \left(2\gamma|\Omega|  p_{\infty}\|g\|_\infty+|\Omega|A\right)\int_0^T\hspace{-0.2cm}\int|N^\varepsilon-\bar{N}|\,dx\,dt.
	\end{split}
	\end{equation*}
	Let $\beta_2\coloneqq2\gamma|\Omega|p_{\infty}\|g\|_\infty+|\Omega|A$. If $\beta_2\|g\|_\infty \left\|\tfrac{\partial p}{\partial S}\right\|_\infty<1$, from \eqref{N-N0} we obtain
	\begin{equation}
	\label{S-S0}
	\int_0^T\hspace{-0.2cm}\int|S^\varepsilon-\bar{S}|\,dx\,dt\le\frac{\beta_2 p_\infty}{1-\beta_2\|g\|_\infty \left\|\tfrac{\partial p}{\partial S}\right\|_\infty}\int_0^T\hspace{-0.2cm}\iint_0^\infty|n^\varepsilon-\bar{n}|\,ds\,dx\,dt.
	\end{equation}
	Let $\beta_3\coloneqq\frac{\beta_2 p_\infty}{1-\beta_2\|g\|_\infty \left\|\tfrac{\partial p}{\partial S}\right\|_\infty}$, then from \eqref{n-n0} we deduce the following inequality
	\begin{equation}
	\begin{split}
	\int_0^T\hspace{-0.2cm}\iint_0^\infty p(s,S^\varepsilon)|n^\varepsilon-\bar{n}|\,ds\,dy\,dt&\le\varepsilon\int_0^T\hspace{-0.2cm}\iint_0^\infty|\partial_tn^\varepsilon|\,ds\,dy\,dt+2\beta_3\int_0^T\hspace{-0.2cm}\iint_0^\infty|n^\varepsilon-\bar{n}|\,ds\,dy\,dt\\
	&\quad+\int_0^T\hspace{-0.2cm}\int\left|\int_0^\infty p(s,S^\varepsilon)(n^\varepsilon-\bar{n})\,ds\right|dy\,dt.
	\end{split}
	\end{equation}
	Since $\int_0^\infty(n^\varepsilon-\bar{n})\,ds=0$ and $p\ge p_{*}$, we have the following inequality
	\begin{equation*}
	\begin{split}
		\int_0^T\hspace{-0.2cm}\int\left|\int_0^\infty p(s,S^\varepsilon)(n^\varepsilon-\bar{n})\,ds\right|dy\,dt&=\int_0^T\hspace{-0.2cm}\int\left|\int_0^\infty (p(s,S^\varepsilon)-p_{*})(n^\varepsilon-\bar{n})\,ds\right|dy\,dt\\
		&\le\int_0^T\hspace{-0.2cm}\int\int_0^\infty (p(s,S^\varepsilon)-p_{*})|n^\varepsilon-\bar{n}|\,ds\,dy\,dt.
	\end{split}
	\end{equation*}
	As $\gamma$ is small enough, we finally conclude the following Poincaré-like estimate for $n^\varepsilon-\bar{n}$
	\begin{equation}
	\left(p_{*}-2\beta_3\right)\int_0^T\hspace{-0.2cm}\iint_0^\infty |n^\varepsilon-\bar{n}|\,ds\,dy\,dt\le\varepsilon\int_0^T\hspace{-0.2cm}\iint_0^\infty|\partial_tn^\varepsilon|\,ds\,dy\,dt.
	\end{equation}
	And we obtain the result by taking $\varepsilon\to0$, since the $L^1$ norm of $\partial_tn^\varepsilon$ is uniformly bounded in $\varepsilon$. The convergence of $N,S$ and $w$ is straightforward from estimates \eqref{N-N0}, \eqref{w-w0} and \eqref{S-S0}.
\end{proof}

\begin{rmk}
	For a firing rate $p$ satisfying \eqref{lbp2} is not evident to apply Doeblin's theory to deduce theorem \ref{conveps}. Indeed, for a fixed $S\in\C_b(\Omega)$, consider $P^\varepsilon_t$ the semi-group defined by the linear problem
	\begin{equation}
	\left\{
	\begin{matrix*}[l]
	\varepsilon\partial_t n+\partial_s n+p(s,S(x))n=0&t>0,s>0,x\in\Omega\vspace{0.15cm},\\
	N(t,x)\coloneqq n(t,s=0,x)=\int_0^\infty p(s,S(x))n\,ds& t>0,x\in\Omega\vspace{0.15cm},\\
	n(t=0,s,x)=n_0(s,x)&s\ge0,x\in\Omega,
	\end{matrix*}
	\right.
	\end{equation}
	so that by replicating the proof of theorem \ref{DoeblinConv} we can prove the following lower bound
	$$P^\varepsilon_{t_0}n_0\ge \varepsilon p_*e^{-2p_{\infty}s_*}\,\mathds{1}_{[0,s_*]}(s)\,g(x),\quad t_0=2\varepsilon s_*.$$
	And we lose Doeblin's condition as $\varepsilon$ vanishes. 
\end{rmk}

\section{Numerical results}
\label{numerical}
\subsection{Elapsed time dynamics}
We present numerical simulations of the system \eqref{eq0} in order to observe the dependence on parameters like connectivity $\gamma$ and the input $I$. For these simulations the domain for position $x$ is $\Omega=(0,1)$ and the firing rate is given by $p=\mathds{1}_{\{ s>S\}}$. We compute numerical solutions with a standard upwind scheme.

We focus in displaying the activity $N(t,x)$ and the amplitude $S(t,x)$ since these two elements determine the general behavior of system \eqref{eq0}. We explore a spatially-homogeneous case and an inhomogeneous one, both with a different learning rule for $w$. In every example the initial connectivity kernel is given by $$w_0(x,y)=10\exp\left(-10(x-y)^2\right).$$

\subsubsection{Spatially-homogeneous input}
We start with some examples when the external input $I$ is constant and positive. For this sub-section the initial probability density is given by $n_0(s,x)=(x+1)e^{-s(x+1)}$, so that $g\equiv1$. Moreover, we consider a learning rule of Hebbian type with the evolution of the kernel given by
$$\partial_t w=-w+\gamma N(t,x)N(t,y).$$
In this particular example there exists a unique steady state determined, through the formulas in \eqref{solest}, by a unique amplitude of stimulation $S^*$, which is constant. This is given by the unique positive solution of the equation
$$S=\frac{\gamma}{(1+S)^3}+I.$$

\begin{figure}[ht!]
	\centering
	\caption{Case $\gamma=1$ and $I=1$.}
	\label{g1i1c}
	\begin{subfigure}{0.33\textwidth}
		\includegraphics[width=\textwidth]{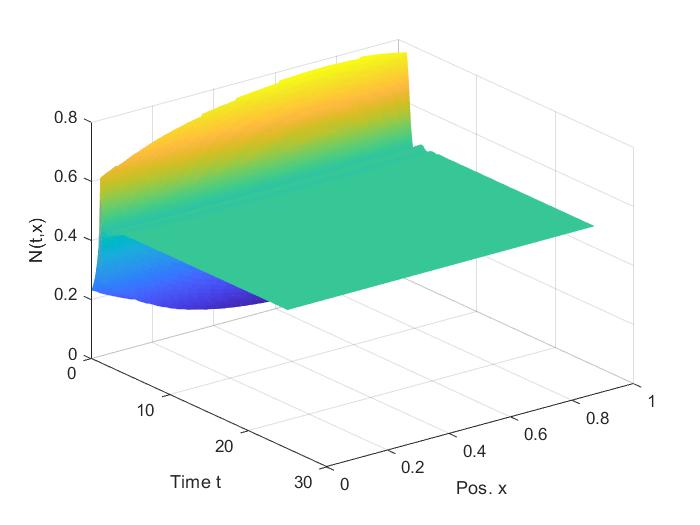}
		\caption{Activity $N(t,x)$.}
	\end{subfigure}  
	\begin{subfigure}{0.33\textwidth}
		\includegraphics[width=\textwidth]{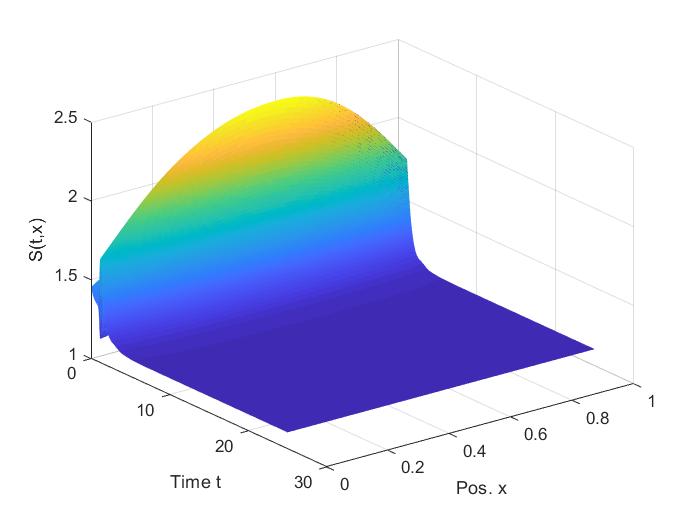}
		\caption{Amplitude of stimulation $S(t,x)$.}
	\end{subfigure}
	\begin{subfigure}{0.33\textwidth}
		\includegraphics[width=\linewidth]{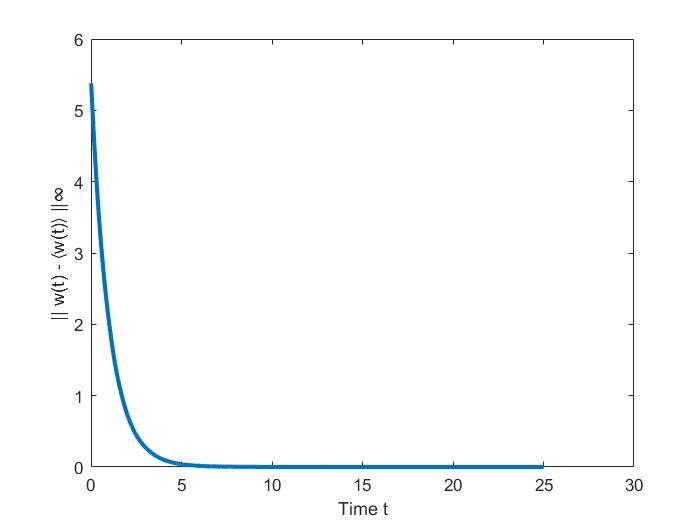}
		\caption{Variation of $\|w(t)-\langle w(t)\rangle\|_\infty$.}
		\label{wg1i1c}
	\end{subfigure}
\end{figure}

In figure \ref{g1i1c} we observe that for $\gamma=1$ and $I=1$ the activity $N$ and the amplitude $S$ stabilize very fast in time and become spatially-homogeneous, this means that the numerical solution $n$ of the system \eqref{eq0} converges to the equilibrium which is independent of variable $x$. Moreover, we observe \ref{wg1i1c} that $\|w(t)-\langle w(t)\rangle\|_\infty$, with $\langle w \rangle\coloneqq|\Omega|^{-2}\iint w\,dx\,dy$, decreases to $0$ in time so the numerical connectivity kernel $w$ is converging to a constant. We essentially observe the behavior of theorem \ref{conveq1}.

\begin{figure}[ht!]
	\centering
	\caption{Case $\gamma=15$ and $I=1$.}
	\label{g15i1c}
	\begin{subfigure}{0.33\textwidth}
		\includegraphics[width=\textwidth]{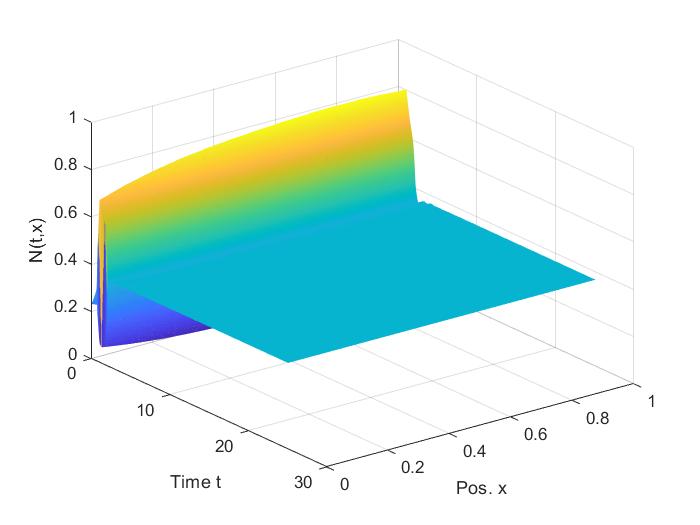}
		\caption{Activity $N(t,x)$.}
	\end{subfigure}  
	\begin{subfigure}{0.33\textwidth}
		\includegraphics[width=\textwidth]{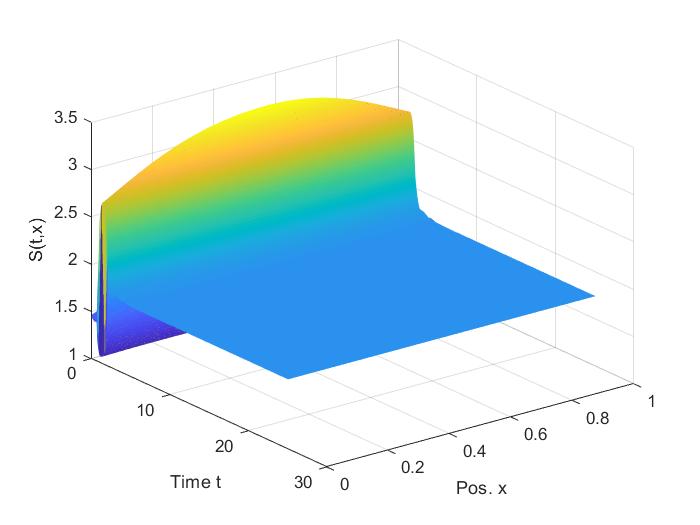}
		\caption{Amplitude of stimulation $S(t,x)$.}
	\end{subfigure}
	\begin{subfigure}{0.33\textwidth}
		\includegraphics[width=\linewidth]{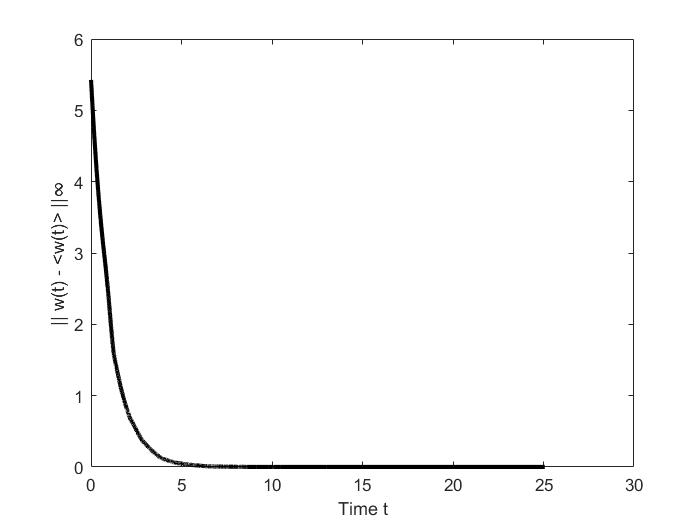}
		\caption{Variation of $\|w(t)-\langle w(t)\rangle\|_\infty$.}
		\label{wg15i1c}
	\end{subfigure}
\end{figure}

If we increase the value of  to $\gamma=15$, we observe in figure \ref{g15i1c} that $N$ and $S$ converge also converge to a steady-state and they become spatially-homogeneous. We observe in figure \ref{wg15i1c} that $\|w(t)-\langle w(t)\rangle\|_\infty$ decreases to $0$ with time, so the connectivity kernel $w$ is converging to a spatially-homogeneous pattern as well.

If we take $\gamma=35$ and also increase the value of input $I$, the numerical solution exhibits again convergence towards equilibrium when the time is large enough. Like the previous cases, the activity $N$ and the amplitude $S$ become spatially-homogeneous in figure \ref{g35i5c}. For the connectivity kernel we have that $\|w(t)-\langle w(t)\rangle\|_\infty$ decreases to $0$ in time as we observe in figure \ref{wg35i5c}, so the numerical connectivity $w$ is converging to a constant. Moreover, this is also compatible with the large connectivity case studied in the article of Pakdaman et al. \cite{PPD}.

\begin{figure}[ht!]
	\centering
	\caption{Case $\gamma=35$ and $I=5$.}
	\label{g35i5c}
	\begin{subfigure}{0.33\textwidth}
		\includegraphics[width=\textwidth]{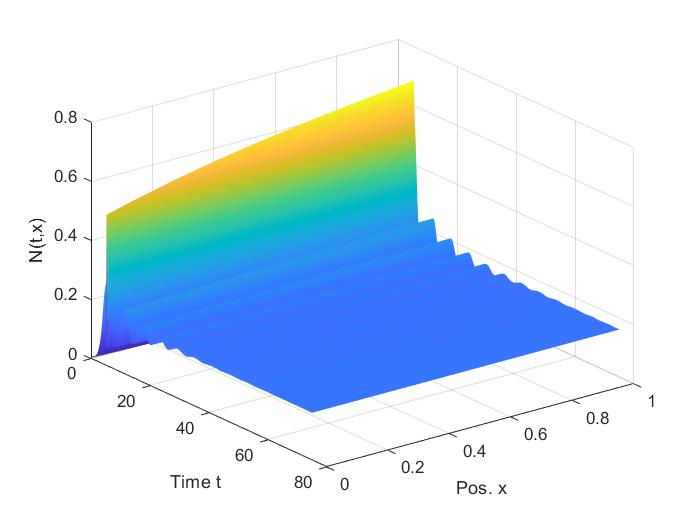}
		\caption{Activity $N(t,x)$.}
	\end{subfigure}  
	\begin{subfigure}{0.33\textwidth}
		\includegraphics[width=\textwidth]{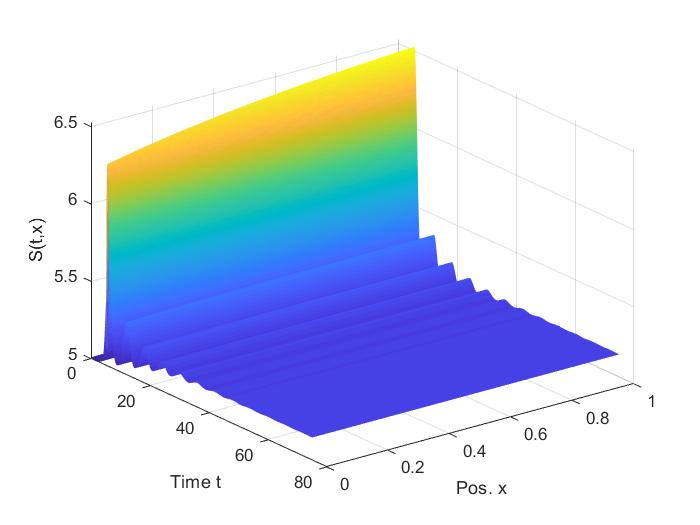}
		\caption{Amplitude of stimulation $S(t,x)$.}
	\end{subfigure}
	\begin{subfigure}{0.33\textwidth}
		\includegraphics[width=\linewidth]{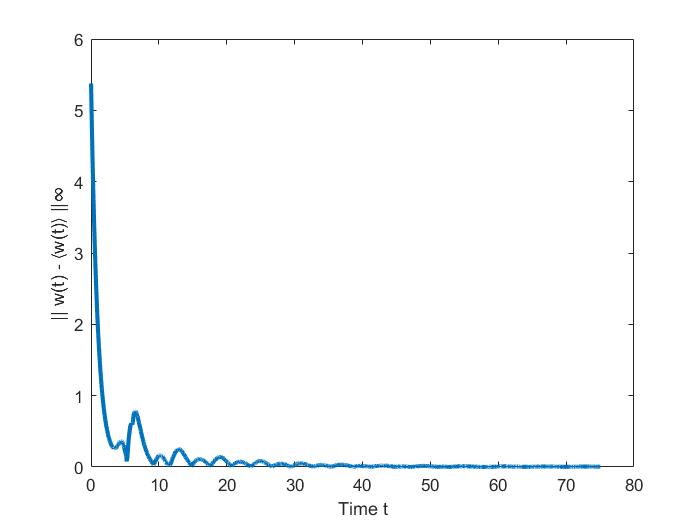}
		\caption{Variation of $\|w(t)-\langle w(t)\rangle\|_\infty$.}
		\label{wg35i5c}
	\end{subfigure}
\end{figure}

More generally, we can conjecture that when $g$ and the input $I$ are constant, then $N,S$ and $w$ lose its spatial dependence as time passes.

\subsubsection{A spatially-inhomogeneous example}
\subsubsection{Spatially-inhomogeneous input}
Now we present an example with a non-constant input to see the activity and the connectivity kernel depending strongly on position. For this subsection the initial probability density is given by $n_0(s,x)=\frac{\exp\left(-s-(x-0.5)^2\right)}{\int_0^1 \exp(-(z-0.5)^2)\,dz}$. We consider a learning rule with the evolution of the kernel given by
$$\partial_t w=-w+\gamma\frac{\exp\left(-(N(t,x)-N(t,y))^2\right)}{1+\exp\left(-2N(t,x)N(t,y)+2\right)}.$$

Consider first $I(x)=\sin^2(2\pi x)$, so for $\gamma=1$ we observe in figure \ref{g1i1v} that both $N$ and $S$ converge in time to a stationary state. Moreover in figure \ref{wg1i1v}, we observe that the connectivity kernel converges to a particular pattern that exhibits a symmetric behavior in spatial variable. Like the corresponding spatially-homogeneous example of figure \ref{g1i1c}, we observe again the behavior of theorem \ref{conveq1}. 
\begin{figure}[ht!]
	\centering
	\caption{Case $\gamma=1$ and $I=\sin^2(2\pi x)$.}
	\label{g1i1v}
	\begin{subfigure}{0.33\textwidth}
		\includegraphics[width=\textwidth]{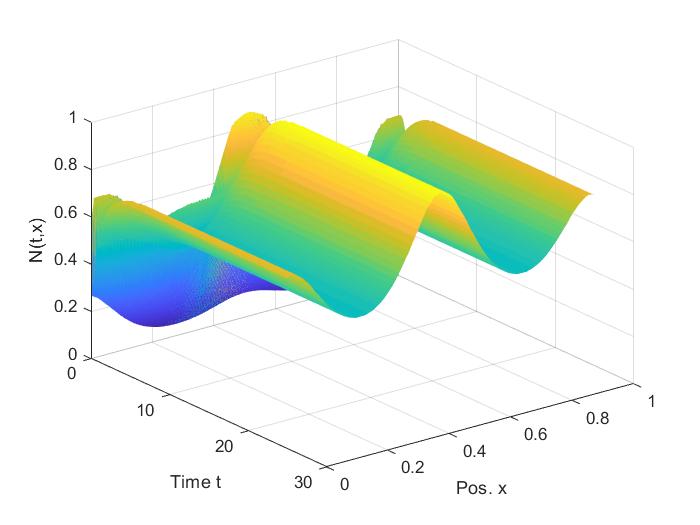}
		\caption{Activity $N(t,x)$.}
	\end{subfigure}  
	\begin{subfigure}{0.33\textwidth}
		\includegraphics[width=\textwidth]{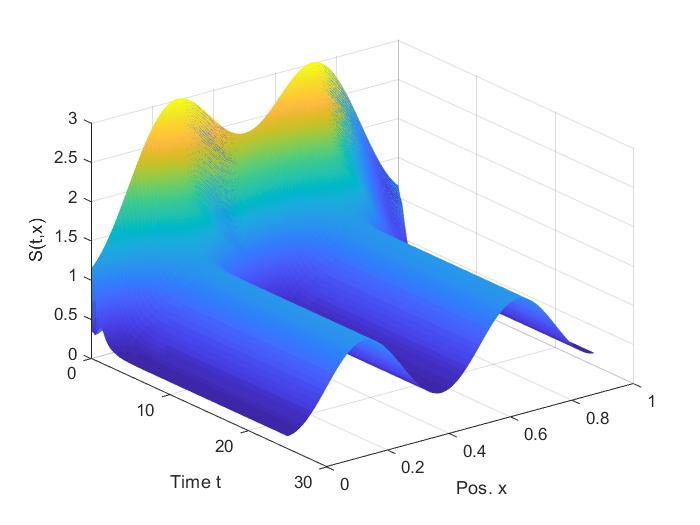}
		\caption{Amplitude of stimulation $S(t,x)$.}
	\end{subfigure}
	\begin{subfigure}{0.33\textwidth}
		\includegraphics[width=\textwidth]{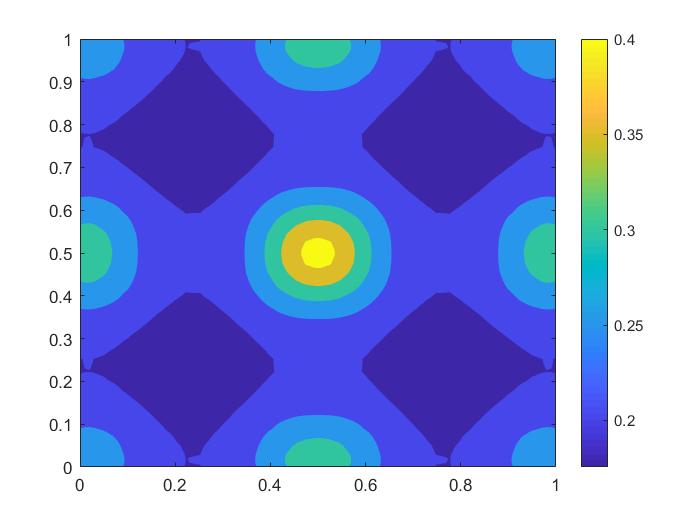}
		\caption{Connectivity $w(t,x,y)$ at $t=25$.}
		\label{wg1i1v}
	\end{subfigure}
\end{figure}

\begin{figure}[ht!]
	\centering
	\caption{Case $\gamma=10$ and $I=\sin^2(2\pi x)$.}
	\label{g10i1v}
	\begin{subfigure}{0.33\textwidth}
		\includegraphics[width=\textwidth]{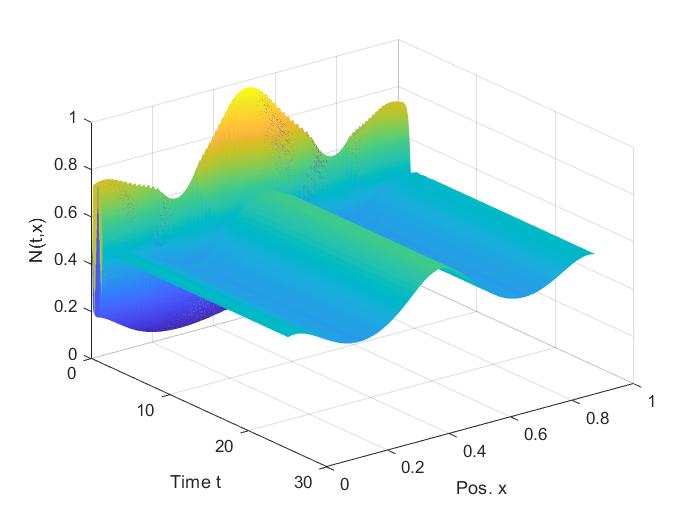}
		\caption{Activity $N(t,x)$.}
	\end{subfigure}  
	\begin{subfigure}{0.33\textwidth}
		\includegraphics[width=\textwidth]{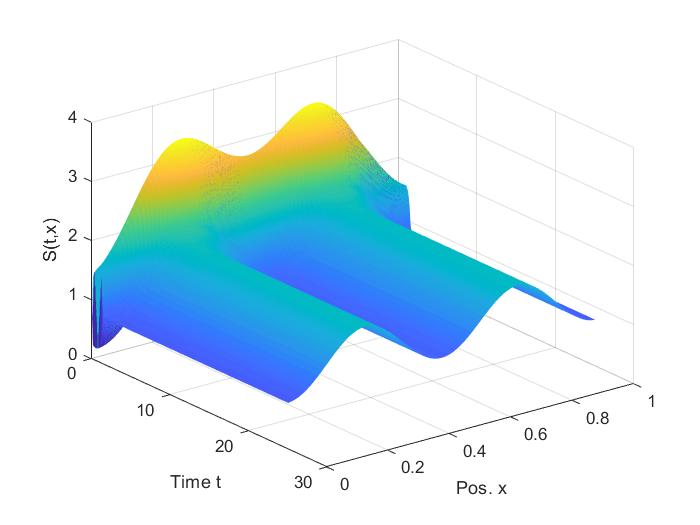}
		\caption{Amplitude of stimulation $S(t,x)$.}
	\end{subfigure}
	\begin{subfigure}{0.33\textwidth}
		\includegraphics[width=\textwidth]{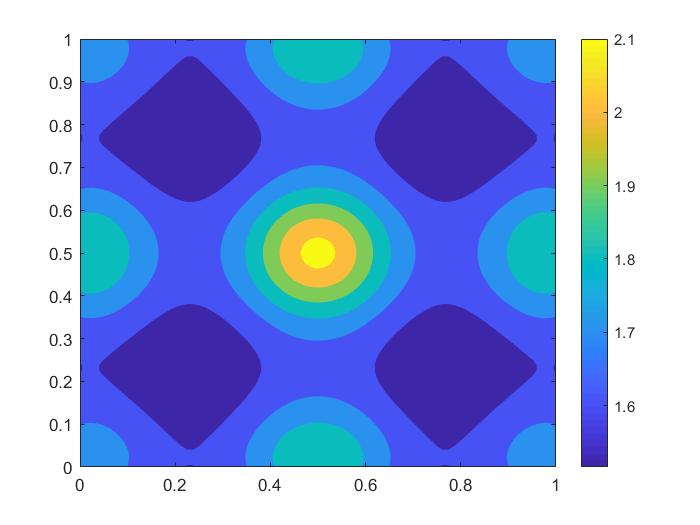}
		\caption{Connectivity $w(t,x,y)$ at $t=25$.}
		\label{wg10i1v}
	\end{subfigure}
\end{figure}

As in the previous example, if we increase the connectivity parameter to $\gamma=10$, the behavior of the activity $N$ and the amplitude $S$ is essentially the same, as we can see in figure \ref{g10i1v}. The connectivity kernel converge the pattern shown in figure \ref{wg10i1v} and it presents higher values than those in figure \ref{wg1i1v}.

\begin{figure}[ht!]
	\centering
	\caption{Case $\gamma=20$ and $I=5\sin^2(2\pi x)$.}
	\label{g20i5v}
	\begin{subfigure}{0.33\textwidth}
		\includegraphics[width=\textwidth]{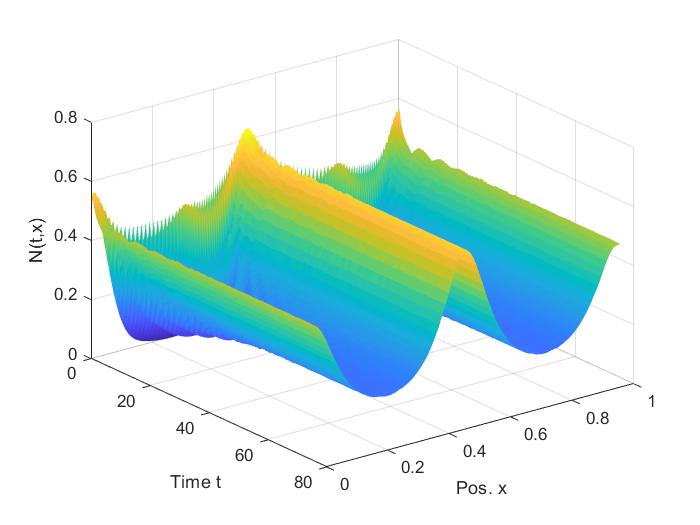}
		\caption{Activity $N(t,x)$.}
	\end{subfigure}  
	\begin{subfigure}{0.33\textwidth}
		\includegraphics[width=\textwidth]{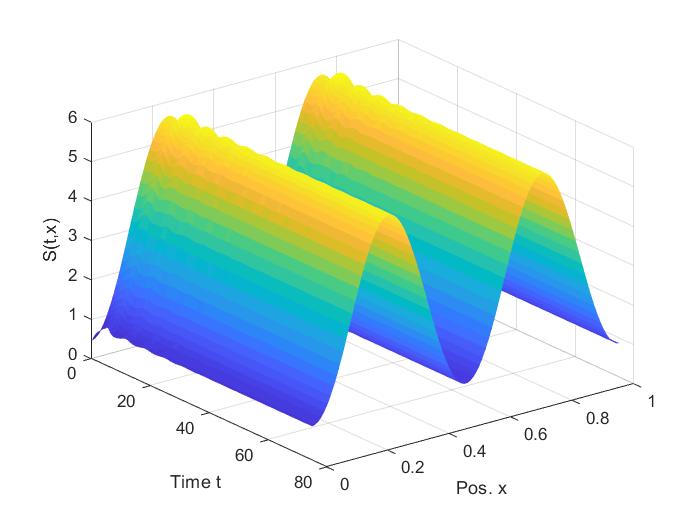}
		\caption{Amplitude of stimulation $S(t,x)$.}
	\end{subfigure}
	\begin{subfigure}{0.33\textwidth}
		\includegraphics[width=\textwidth]{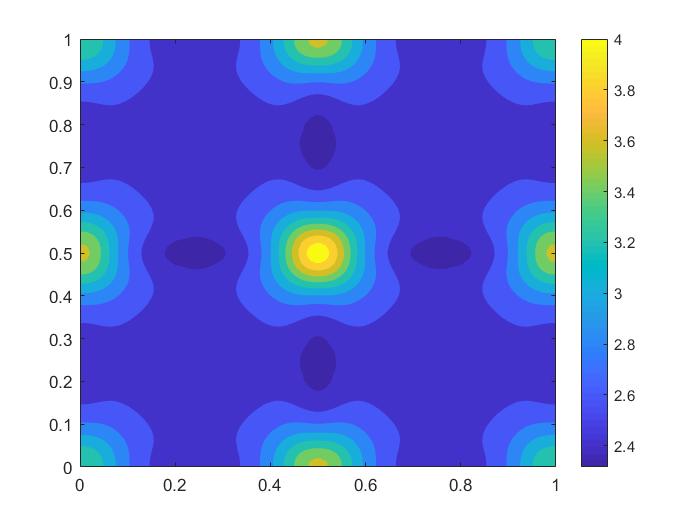}
		\caption{Connectivity $w(t,x,y)$ at $t=75$.}
		\label{wg20i5v}
	\end{subfigure}
\end{figure}

Finally, in the case of $\gamma=20$ and $I=5\sin^2(2\pi x)$, the numerical solution exhibits convergence towards an equilibrium when the time is large enough as it is presented in figure \ref{g20i5v}. The numerical connectivity kernel $w$ converge to pattern presented in figure \ref{wg20i5v}.

From these examples, for both spatially-homogeneous and inhomogeneous cases, we conjecture that if the system is inhibitory, then all solutions of system \eqref{eq0} converge to a steady-state. This result is also conjectured for the classical elapsed-time model studied in \cite{PPD}.

Moreover when the input $I$ is large enough, we expect a similar convergence result. Theorem \ref{largeinput} states that solutions converge pointwise to a solution of a simple linear problem when the external input is large enough in both spatially-homogeneous and inhomogeneous cases. This theorem could be a first approach to verify the general convergence result.

\subsection{Limit system with $\varepsilon=0$.}

We present some numerical simulations of the limit system \eqref{eqeps0} under the same setting of domain, firing rate and initial kernel. We show the homogeneous and inhomogeneous cases with the same respective initial densities, learning rules and parameter combinations of their counterparts of system \eqref{eq0}. We contrast the numerical simulations with the convergence theorem \ref{conveps} when $\varepsilon$ vanishes.

\subsubsection{Spatially-homogeneous input}
In figure \ref{Lg1i1c} we observe that for $\gamma=1$ and $I=1$ both $N,S$ converge fast in time to equilibrium and become spatially-homogeneous. Moreover the figure \ref{Lwg1i1c} shows that $\|w(t)-\langle w(t)\rangle\|_\infty$ is converging to $0$, so $w$ is converges to a constant. This corresponds essentially to the same behavior of the numerical simulations in system \eqref{eq0} and it is compatible with the convergence result of theorem \ref{conveps}.

\begin{figure}[ht!]
	\centering
	\caption{Case $\gamma=1$ and $I=1$ for the limit system.}
	\label{Lg1i1c}
	\begin{subfigure}{0.33\textwidth}
		\includegraphics[width=\textwidth]{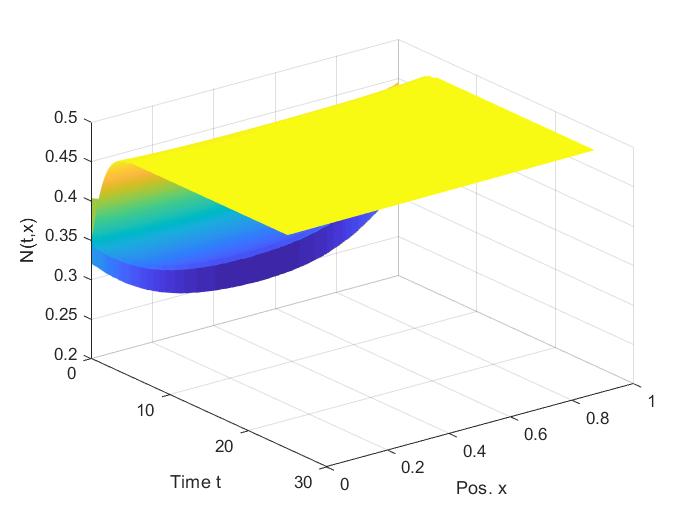}
		\caption{Activity $N(t,x)$.}
	\end{subfigure}  
	\begin{subfigure}{0.33\textwidth}
		\includegraphics[width=\textwidth]{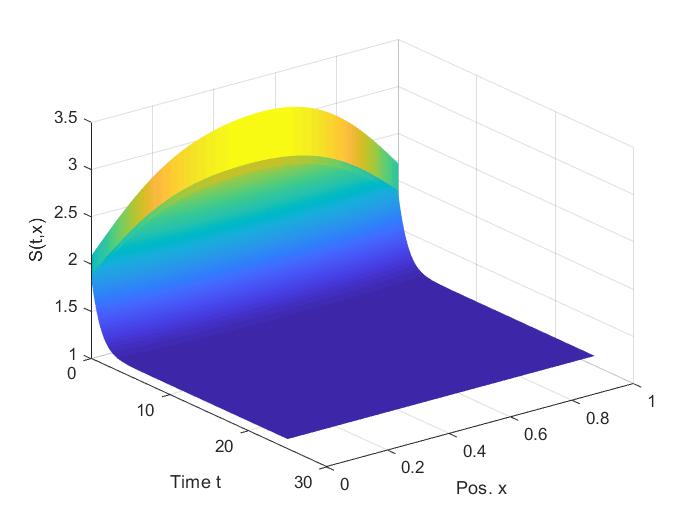}
		\caption{Amplitude of stimulation $S(t,x)$.}
	\end{subfigure}
	\begin{subfigure}{0.33\textwidth}
	\includegraphics[width=\textwidth]{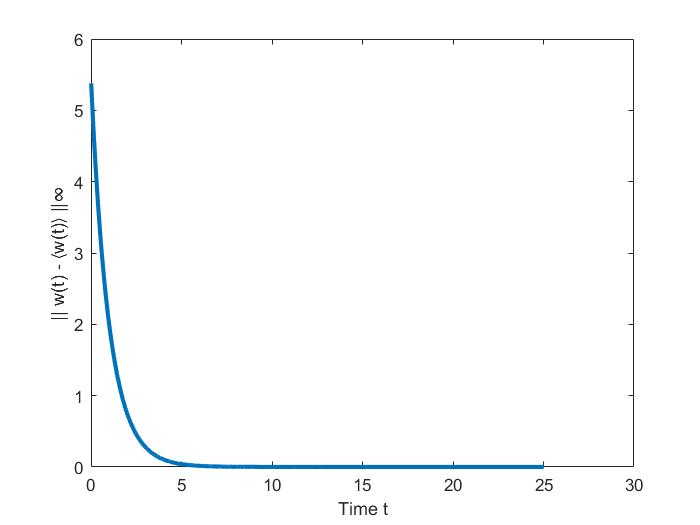}
	\caption{Variation of $\|w(t)-\langle w(t)\rangle\|_\infty$.}
	\label{Lwg1i1c}
\end{subfigure}
\end{figure}

When we increase the value to $\gamma=15$ numerical solutions keep the same behavior of convergence to equilibrium and spatial homogeneity as we see in figure \ref{Lg15i1c}. From figure \ref{Lwg15i1c} we observe that the numerical connectivity kernel verifies that $\|w(t)-\langle w(t)\rangle\|_\infty$ is converging to $0$ and $w$ converges to a constant.

\begin{figure}[ht!]
	\centering
	\caption{Case $\gamma=15$ and $I=1$ for the limit system.}
	\label{Lg15i1c}
	\begin{subfigure}{0.33\textwidth}
		\includegraphics[width=\textwidth]{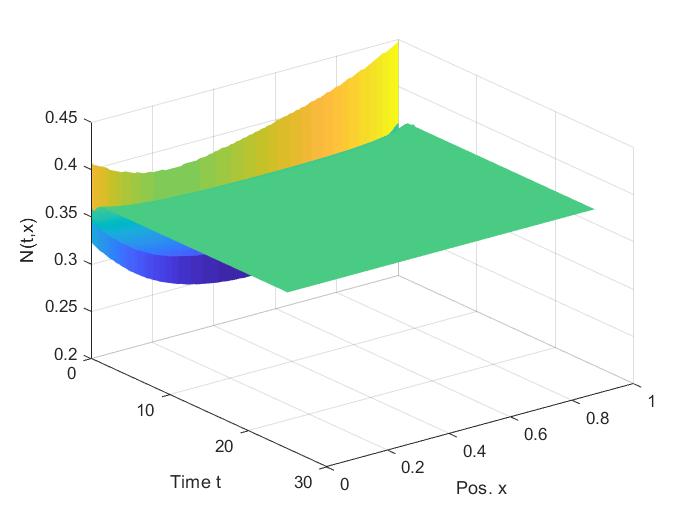}
		\caption{Activity $N(t,x)$.}
	\end{subfigure}  
	\begin{subfigure}{0.33\textwidth}
		\includegraphics[width=\textwidth]{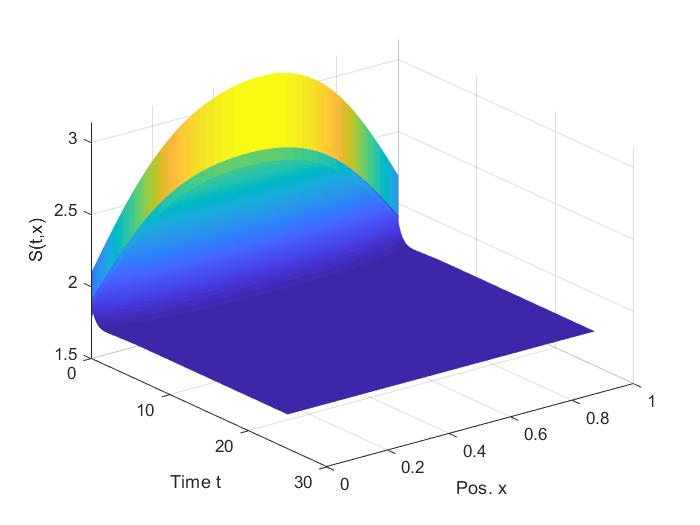}
		\caption{Amplitude of stimulation $S(t,x)$.}
	\end{subfigure}
	\begin{subfigure}{0.33\textwidth}
		\includegraphics[width=\linewidth]{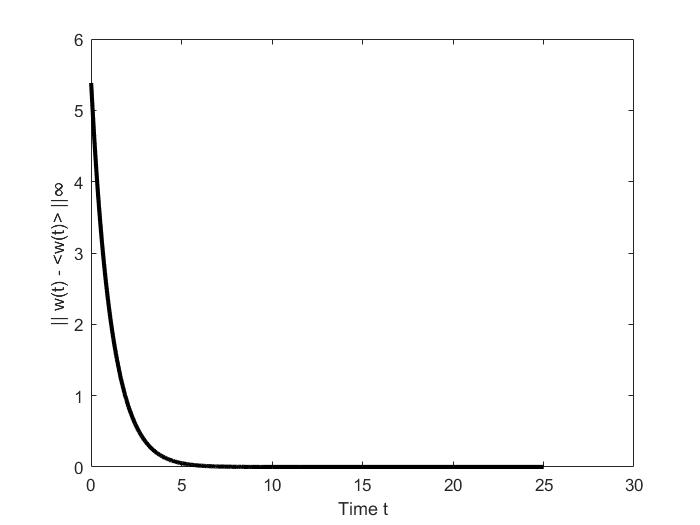}
		\caption{Variation of $\|w(t)-\langle w(t)\rangle\|_\infty$.}
		\label{Lwg15i1c}
	\end{subfigure}
\end{figure}

\begin{figure}[ht!]
	\centering
	\caption{Case $\gamma=35$ and $I=5$ for the limit system.}
	\label{Lg35i5c}
	\begin{subfigure}{0.33\textwidth}
		\includegraphics[width=\textwidth]{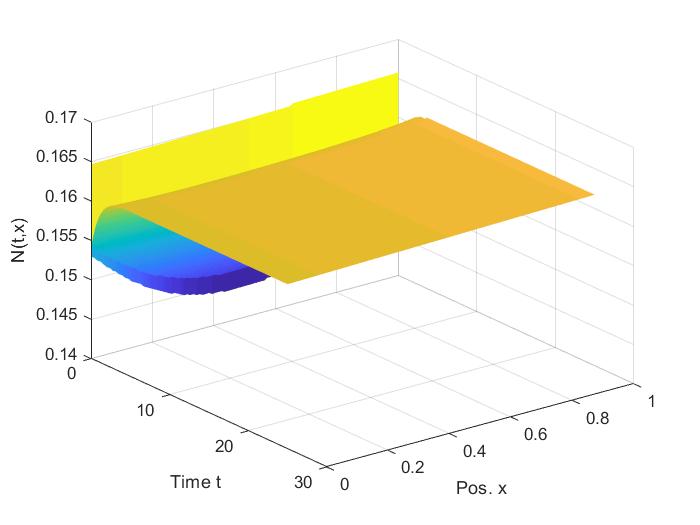}
		\caption{Activity $N(t,x)$.}
	\end{subfigure}  
	\begin{subfigure}{0.33\textwidth}
		\includegraphics[width=\textwidth]{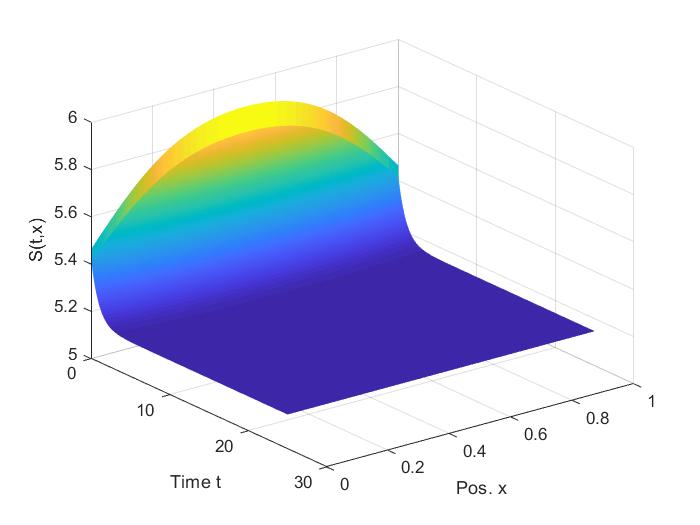}
		\caption{Amplitude of stimulation $S(t,x)$.}
	\end{subfigure}
	\begin{subfigure}{0.33\textwidth}
	\includegraphics[width=\linewidth]{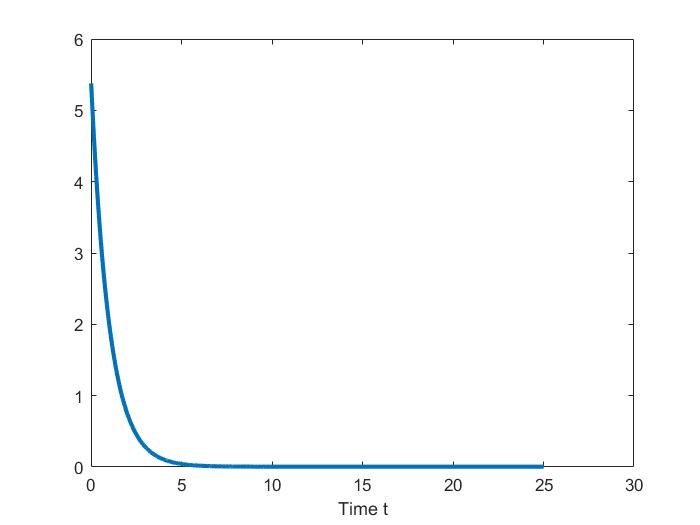}
	\caption{Variation of $\|w(t)-\langle w(t)\rangle\|_\infty$.}
	\label{Lwg35i5c}
\end{subfigure}
\end{figure}

If in addition we take $\gamma=35$ and increase the value of input to $I=5$, we observe in figure \ref{Lg35i5c} the same behavior for $N,S$ and $w$ as in previous cases. Therefore we can conjecture that when $g$ and the input $I$ are constant then the system \eqref{eqeps0} simply converges to a spatially-homogeneous equilibrium, like we observed in the corresponding numerical simulations of system \eqref{eq0}.

\subsubsection{Spatially-inhomogeneous input}

Now we show some numerical simulations of the system \eqref{eqeps0} under the same previously presented non-constant inputs.

If $I=\sin^2(2\pi x)$ and $\gamma=1$ we see in figure \ref{Lg1i1v} that both $N$ and $S$ converge in time to a stationary state as expected. With respect to the kernel $w$, we observe in figure \ref{Lwg1i1v} a similar pattern formation as in the corresponding simulation for the system \eqref{eq0} in figure \ref{wg1i1v}. Furthermore, this example is compatible with the result of theorem \ref{conveps}.

\begin{figure}[ht!]
	\centering
	\caption{Case $\gamma=1$ and $I=\sin^2(2\pi x)$ for the limit system.}
	\label{Lg1i1v}
	\begin{subfigure}{0.33\textwidth}
		\includegraphics[width=\textwidth]{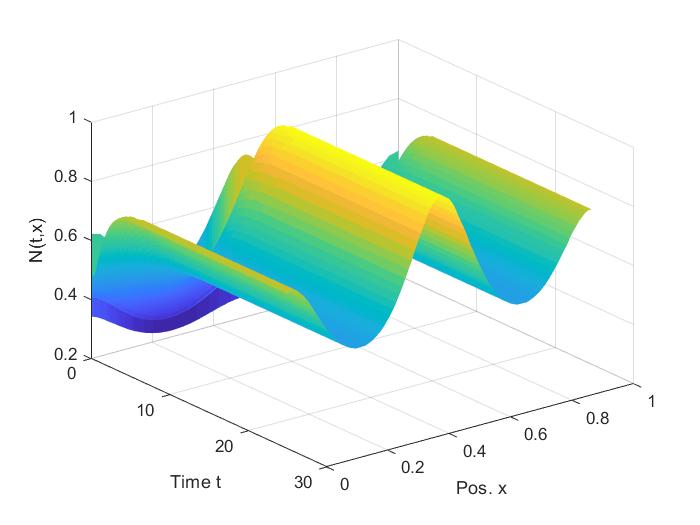}
		\caption{Activity $N(t,x)$.}
	\end{subfigure}  
	\begin{subfigure}{0.33\textwidth}
		\includegraphics[width=\textwidth]{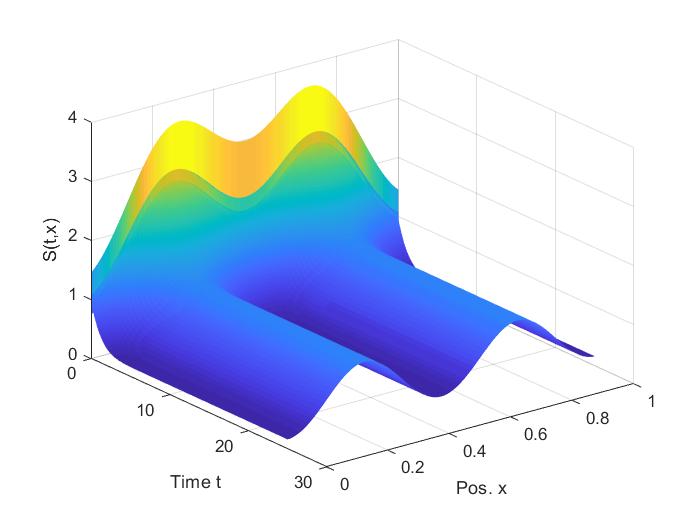}
		\caption{Amplitude of stimulation $S(t,x)$.}
	\end{subfigure}
	\begin{subfigure}{0.33\textwidth}
		\includegraphics[width=\textwidth]{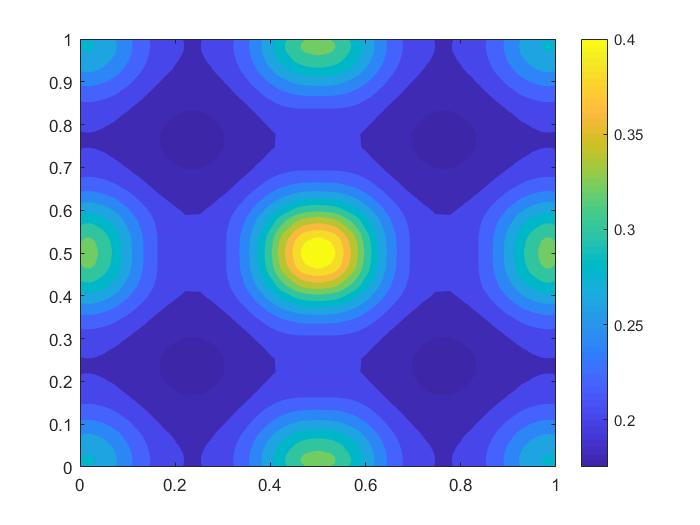}
		\caption{Connectivity $w(t,x,y)$ at $t=25$.}
		\label{Lwg1i1v}
	\end{subfigure}
\end{figure}

\begin{figure}[ht!]
	\centering
	\caption{Case $\gamma=10$ and $I=\sin^2(2\pi x)$ for the limit system.}
	\label{Lg10i1v}
	\begin{subfigure}{0.33\textwidth}
		\includegraphics[width=\textwidth]{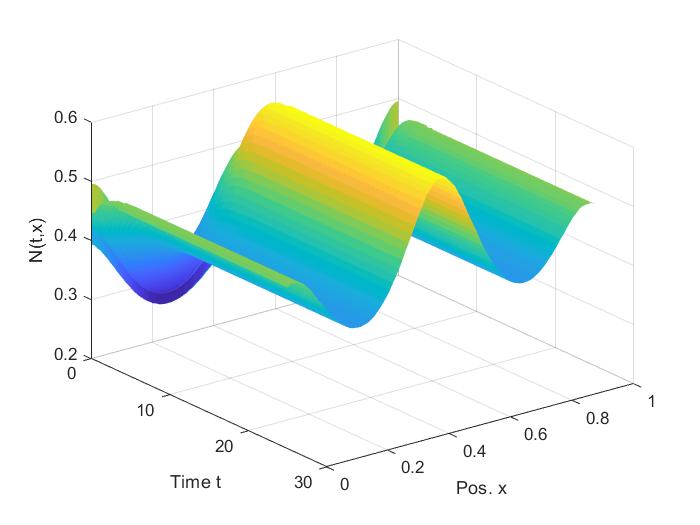}
		\caption{Activity $N(t,x)$.}
	\end{subfigure}  
	\begin{subfigure}{0.33\textwidth}
		\includegraphics[width=\textwidth]{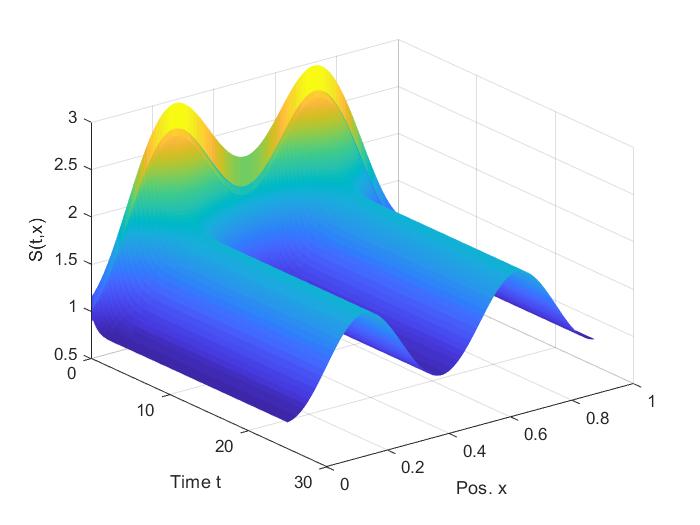}
		\caption{Amplitude of stimulation $S(t,x)$.}
	\end{subfigure}
	\begin{subfigure}{0.33\textwidth}
		\includegraphics[width=\textwidth]{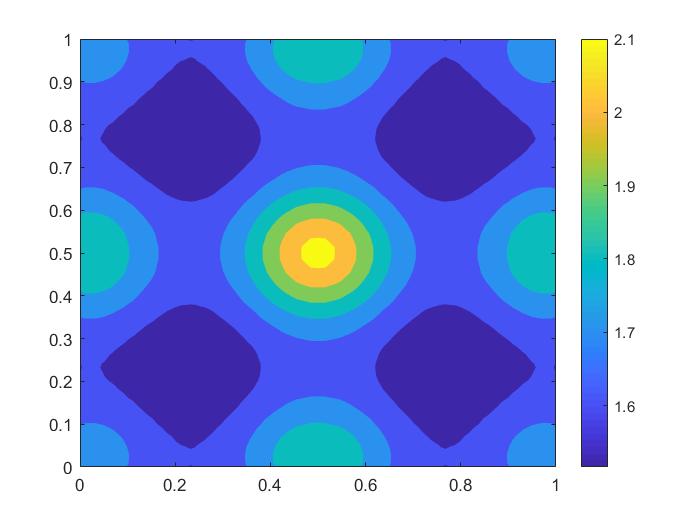}
		\caption{Connectivity $w(t,x,y)$ at $t=25$.}
		\label{Lwg10i1v}
	\end{subfigure}
\end{figure}

\begin{figure}[ht!]
	\centering
	\caption{Case $\gamma=20$ and $I=5\sin^2(2\pi x)$ for the limit system.}
	\label{Lg20i5v}
	\begin{subfigure}{0.33\textwidth}
		\includegraphics[width=\textwidth]{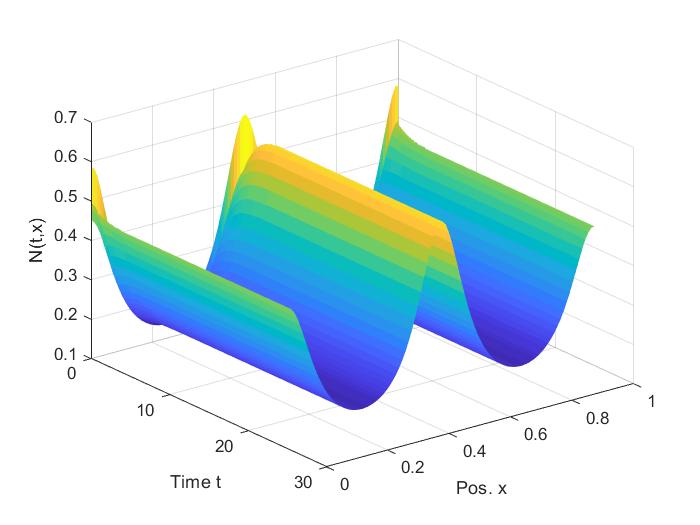}
		\caption{Activity $N(t,x)$.}
	\end{subfigure}  
	\begin{subfigure}{0.33\textwidth}
		\includegraphics[width=\textwidth]{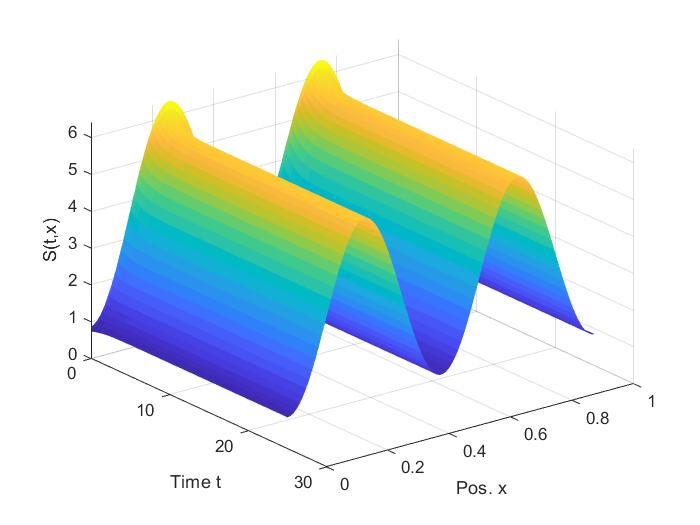}
		\caption{Amplitude of stimulation $S(t,x)$.}
	\end{subfigure}
	\begin{subfigure}{0.33\textwidth}
		\includegraphics[width=\textwidth]{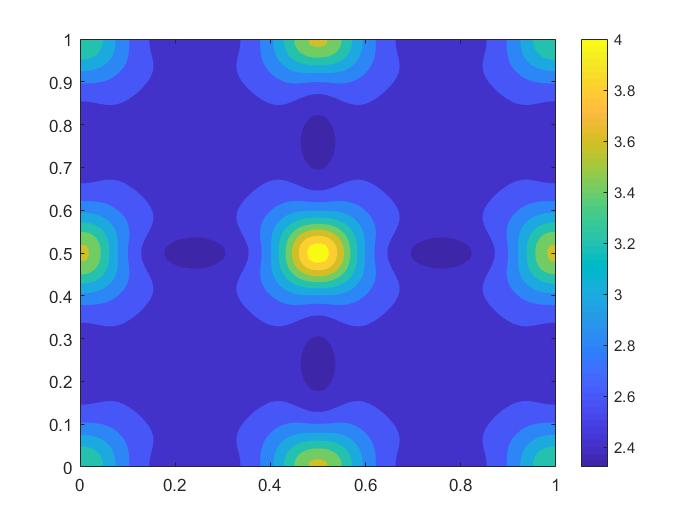}
		\caption{Connectivity $w(t,x,y)$ at $t=25$.}
		\label{Lwg20i5v}
	\end{subfigure}
\end{figure}

Next, when we increase the value to $\gamma=10$, we still observe in figure \ref{Lg10i1v} the convergence in time for $N$ and $S$. Furthermore, the numerical kernel $w$ exhibits in figure \ref{Lwg10i1v} a similar pattern to that observed in figure \ref{wg10i1v}, the corresponding simulation of system \eqref{eq0}.

Finally in the case of $\gamma=20$ and $I=5\sin^2(2\pi x)$, we observe in figure \ref{Lg20i5v} that the numerical solutions exhibits again a convergent behavior in the variables $N$ and $S$, while the kernel shows essentially in figure \ref{Lwg20i5v} the same pattern as the corresponding simulation of the system \eqref{eq0}. We conjecture that the general dynamic of the limit system \eqref{eqeps0} is simply a convergence to stationary state. Furthermore, we conjecture that theorem \ref{conveps} is also true for a strong interconnection in the inhibitory case or for a large external input.

\section{Perspectives}
From the previous theoretical results and numerical simulations we observe that only the case with very weak interconnection begins to be well understood for the Cauchy problem and the asymptotic behavior. More complex dynamics, such as oscillations, that could emerge with stronger interconnections or even convergence to a stationary state for a general case are far from being fully understood.

Concerning well-posedness in the system \eqref{eq0}, it remains unsolved studying the case of a strong connectivity and determine if multiple solutions arise. This means studying the number of solutions for $S(t,x)$ in the fixed point equation in \eqref{fixS}. It also remains open the well-posedness for limit system \eqref{eqeps0} with its corresponding fixed point problem.

Regarding convergence to equilibrium, it is necessary to give a more detailed description of how the size of the kernel $w$ affects the general behavior of system \eqref{eq0} in order to have a clearer idea of the bifurcation diagram in the connectivity parameter $\gamma$.

Furthermore, it is pending to study the convergence to equilibrium of system \eqref{eq0} for a general large input in order to improve theorem \ref{largeinput}. This include to consider the case when the external input $I$ goes to infinity in localized regions of $\Omega$. Moreover, it remains open to prove when the function $g$ and the external input are constant then the system approaches to spatially-homogeneous profile as it was observed in the numerical simulations. 

Finally for the system with slow learning \eqref{eqeps}, we expect the convergence theorem \ref{slowth} for weak interconnection is also true when $p$ satisfies the lower bound \eqref{lbp2}. Furthermore, for the limit system \eqref{eqeps0} we expect a simple convergence to equilibrium  regardless the value of $\gamma$.

\section*{Acknowledgements}
This project has received funding from the European Union's Horizon 2020 research and innovation program under the Marie Sklodowska-Curie grant agreement No 754362.

\begin{center}
    \includegraphics[scale=0.1]{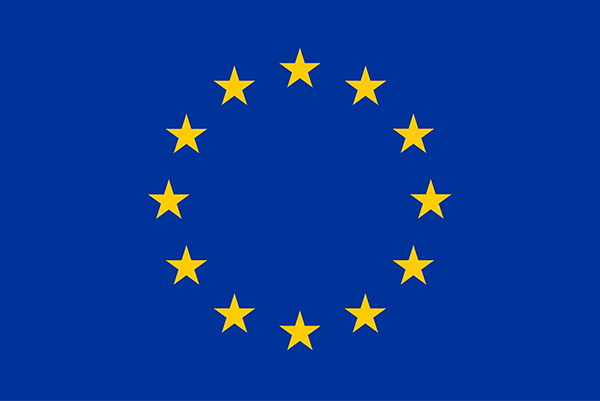}
\end{center}

\bibliography{thesis.bib}
\bibliographystyle{plain}
\nocite{*} 

\end{document}